\documentclass[leqno, 10pt]{amsart}
\usepackage{a4wide}
\usepackage{amsmath, amssymb, latexsym, mathrsfs, mathtools, mathabx, shuffle}

 \newtheorem{theorem}{Theorem}[section]
 \newtheorem{lemma}[theorem]{Lemma}
 \newtheorem{proposition}[theorem]{Proposition}
 \newtheorem{corollary}[theorem]{Corollary}

 \newtheorem*{theorem*}{Theorem}
\newtheorem*{proposition*}{Proposition}
\newtheorem*{lemma*}{Lemma}

\theoremstyle{definition}
\newtheorem{definition}[theorem]{Definition}

 \theoremstyle{remark}
 
 \newtheorem{remark}[theorem]{Remark}

  \newtheorem*{acknowledgements}{Acknowledgements}

\newcommand{\C}{\mathbb{C}}
\newcommand{\N}{\mathbb{N}}
\newcommand{\Q}{\mathbb{Q}}
\newcommand{\Z}{\mathbb{Z}}

\newcommand{\bC}{\mathbf{C}}

\newcommand{\cA}{\mathcal{A}}
\newcommand{\cO}{\mathcal{O}}

\newcommand{\sC}{\mathscr{C}} 
\newcommand{\sM}{\mathscr{M}}

\newcommand{\fS}{\ensuremath{\mathfrak{S}}}

\newcommand{\bt}{\bullet}

\newcommand{\op}[1]{\operatorname{#1}}

\newcommand{\Tot}{\op{Tot}}
\newcommand{\Diag}{\op{Diag}}
\newcommand{\AW}{\op{AW}}

\newcommand{\Mod}{\op{Mod}}

\newcommand{\HC}{\op{HC}}
\newcommand{\bHC}{\op{\mathbf{HC}}}
\newcommand{\HP}{\op{HP}}
\newcommand{\bHP}{\op{\mathbf{HP}}}

\newcommand{\acou}[2]{\ensuremath{\langle #1 , #2 \rangle}} 
\newcommand{\brak}[1]{\ensuremath{\langle #1\rangle}}

\newcommand{\Ch}{\op{Ch}}
\newcommand{\CS}{\op{CS}}

\newcommand{\ev}{{\textup{ev}}}
\newcommand{\odd}{{\textup{odd}}}

\newcommand{\wb}{\widebar{b}}
\newcommand{\wB}{\widebar{B}}
\newcommand{\wC}{\widebar{C}}
\newcommand{\wbd}{\widebar{d}}
\newcommand{\ws}{\widebar{s}}
\newcommand{\wt}{\widebar{t}}
\newcommand{\wT}{\widebar{T}}

\newcommand{\wpi}{\widebar{\pi}}
\newcommand{\wsigma}{{\widebar{\sigma}}} 
\newcommand{\wepsn}{\widebar{\varepsilon}^\natural}

\newcommand{\OG}{{\overline{\Gamma}}}

\newcommand{\ran}{\op{ran}}

\renewcommand{\frown}{\smallfrown}

\numberwithin{equation}{section}

\begin{document}
\title{Cyclic Homology and Group Actions}

 \author{Rapha\"el Ponge}
 \address{Department of Mathematical Sciences, Seoul National University, Seoul, South Korea}
 \email{ponge.snu@gmail.com}

\dedicatory{Dedicated to Alain Connes on the occasion of his 70th birthday}

 \thanks{Research partially supported by Basic Research grants 2013R1A1A2008802 and 2016R1D1A1B01015971 of National Research Foundation (South Korea). }
 
 \keywords{cyclic homology, group homology, equivariant cohomology}

\subjclass[2010]{19D55, 20J06, 55N91}

\begin{abstract}
In this paper we present the construction of explicit quasi-isomorphisms that compute the cyclic homology and periodic cyclic homology of crossed-product algebras associated with (discrete) group actions.  In the first part we deal with algebraic crossed-products associated with group actions on unital algebras over any ring $k\supset \Q$.   In the second part, we extend the results to actions on locally convex algebras. We then deal with crossed-products associated with group actions on manifolds and smooth varieties. For the finite order components, the results are expressed in terms of what we call ``mixed equivariant cohomology". This ``mixed" theory mediates between group homology and de Rham cohomology. It is naturally related to equivariant cohomology, and so we obtain explicit constructions of cyclic cycles out of equivariant characteristic classes. For the infinite order components, we simplify and correct the misidentification of~\cite{Cr:KT99}. An important new homological tool is the notion of ``triangular $S$-module". This is a natural generalization of the cylindrical complexes of Getzler-Jones. It combines the mixed complexes of Burghelea-Kassel and parachain complexes of Getzler-Jones with the $S$-modules of Kassel-Jones. There are spectral sequences naturally associated  with triangular $S$-modules. In particular, this allows us to recover spectral sequences of Feigin-Tsygan and Getzler-Jones and leads us to a new spectral sequence. 
\end{abstract}

\maketitle

\section*{Introduction}\label{sec:Intro}
Cyclic homology was introduced by Connes~\cite{Co:Ober81,Co:CRAS83,Co:NCDG}  as the relevant noncommutative analogue of de Rham theory (see also  Tsygan~\cite{Ts:UMN83}). In particular, Connes~\cite{Co:NCDG} established that the cyclic homology of the algebra of smooth functions on a manifold is precisely given by its de Rham cohomology.  Moreover, the dual theory, cyclic cohomology, is the natural receptacle for the Chern character in $K$-homology (see~\cite{Co:NCDG}). 
In addition, Loday-Quillen~\cite{LQ:CMH84} computed the cyclic homology of the algebra of regular functions on smooth varieties, and Burghelea~\cite{Bu:CMH85} computed the cyclic homology of group rings. 

In the setting of noncommutative geometry, crossed-product algebras play the role of the algebra of functions on the noncommutative spaces associated with group actions. 
There is a great amount of work on the cyclic homology of crossed-product algebras. In particular, their cyclic homology is known in the case of actions of finite groups (see, e.g., \cite{BC:CCDG,BDN:AIM17,FT:LNM87,GJ:Crelle93}). For general group actions, Feigin-Tsygan~\cite{FT:LNM87} constructed a spectral sequence converging to cyclic homology. In addition, Nistor~\cite{Ni:InvM90} described the periodicity operator of the infinite order components in terms of a module structure over group cohomology. In the setting of group actions on manifolds, the case of discrete proper actions was settled by Baum-Connes~\cite{BC:CCDG}. For general group actions, Connes~\cite{Co:Kyoto83,Co:NCG} computed the periodic cyclic homology of the homogenous component in terms of equivariant homology via the construction of an explicit cochain map.  The cyclic homology of the other finite order components was computed by Brylinski-Nistor~\cite{BN:KT94}, once again in terms of equivariant homology. In addition, Crainic~\cite{Cr:KT99} attempted to deal with the infinite order components, but he misidentified the homologies he obtained (see Remark~\ref{rmk:mfld-infinite.Crainic} on this point). More recently, Brodzki-Dave-Nistor~\cite{BDN:AIM17} dealt with actions of finite groups on varieties. For various other related results see~\cite{AK:Crelle03, BG:AENS94, BGJ:Crelle95,  Br:AIF87, Br:Preprint87, CM:CMP98, Da:JNCG13, ENN:ActaM88,  Go:CMP99, Mo:AIM15, MR:AIM11, Ni:InvM93, Ne:JFA88, NPPT:Crelle06, PPTT:CCM11}.  

With the exception of Connes' cochain map, the above-mentioned results do not produce explicit quasi-isomorphisms.  The aim of this paper is to present the computation of the cyclic homology of crossed-product algebras via the construction of explicit quasi-isomorphisms, including for the infinite order components. This provides us with a systematic way to construct cyclic cycles. Furthermore, the arguments use only elementary homological algebra, and so they bypass the difficult homological arguments involved in  previous approaches to the cyclic homology of crossed-product algebras. For group actions on manifolds and smooth varieties we obtain explicit relationships with equivariant cohomology. In particular, we simplify and correct the misidentification of the cyclic homology of the infinite order components in~\cite{Cr:KT99}. 

 In the first part of the paper (Sections~\ref{sec:background}--\ref{sec:infinite-order}), we focus on algebraic crossed-products $\cA_\Gamma=\cA\rtimes \Gamma$ associated with the action of an arbitrary group $\Gamma$ on a unital algebra $\cA$ over a commutative ring $k \supset \Q$. The first step is carried out in Section~\ref{sec:conj-classes}. We use the decomposition of the cyclic module $C(\cA_\Gamma)$ as a direct sum of cyclic submodules $C(\cA_\Gamma)_{[\phi]}$ parametrized by the conjugacy classes $[\phi]$ of elements $\phi \in \Gamma$. There are explicit quasi-isomorphisms between $C(\cA_\Gamma)_{[\phi]}$ and a 
mixed complex $\Tot(C^\phi(\Gamma_\phi, \cA))$ associated with the centralizer $\Gamma_\phi$ of $\phi$. This mixed complex is the total mixed complex of a cylindrical module $C^\phi(\Gamma_\phi,\cA)$ that we construct in Section~\ref{sec:CphiGA}. We are thus reduced to studying this cylindrical module. 

The cylindrical module $C^\phi(\Gamma_\phi, \cA)$ is similar to the cylindrical module of Getzler-Jones~\cite{GJ:Crelle93} (see Remark~\ref{rmk:CphiGA-Getzler-Jones} for the comparison between the two constructions). We can think of $C^\phi(\Gamma_\phi, \cA)$ as the tensor product over $\Gamma_\phi$ of twisted cyclic modules $C^\phi(\Gamma)$ and 
$C^\phi(\cA)$ associated with the action of $\phi$. Functoriality arguments further reduce the study of $C^\phi(\Gamma_\phi, \cA)$ to the separate studies of these twisted cyclic modules. Therefore, the bulk of the second step is the construction of explicit quasi-isomorphisms for $C^\phi(\Gamma_\phi)$ and the understanding of how they can be combined with the twisted cyclic structure $C^\phi(\cA)$ to give rise quasi-isomorphisms for $\Tot(C^\phi(\Gamma_\phi, \cA))$. 

The twisted cyclic modules $C^\phi(\Gamma)$ appeared in the computation of the cyclic homology of group rings by Burghelea~\cite{Bu:CMH85}.  Therefore, we may elaborate on the algebraic approach of Marciniak~\cite{Ma:BCP86} and the perturbation theory of Kassel~\cite{Ka:Crelle90} to construct explicit homotopy equivalences for $C^\phi(\Gamma)$. When $\phi$ has finite order this enables us to get explicit quasi-isomorphisms for the cyclic and periodic complexes of $C(\cA)_{[\phi]}$ (Theorem~\ref{thm:finite-order.HCAGphi}). As a corollary we obtain three spectral sequences converging to the cyclic homology of $C(\cA)_{[\phi]}$ (see Corollary~\ref{cor:finite-order.HCAGphi-spectral-sequences}). One of these spectral sequences refines and simplifies the spectral sequence of Getzler-Jones~\cite{GJ:Crelle93}. Another spectral sequence allows us to recover the spectral sequence of Feigin-Tsygan~\cite{FT:LNM87}. When $\Gamma_\phi$ is finite we actually get quasi-isomorphisms with $\Gamma_\phi$-invariant cyclic and period cyclic complexes (see Theorem~\ref{thm:finite.quasi-isomorphism-CAGphi}). 

When $\phi$ has infinite order, we have a homotopy equivalence between the cyclic complex of $C^\phi(\Gamma_\phi)$ and the group homology chain complex $C(\overline{\Gamma}_\phi)$, where $\overline{\Gamma}_\phi= \Gamma_\phi/ \brak\phi$ is the normalizer of $\phi$.  In addition, the periodicity operator is given by the cap product with the Euler class $e_\phi \in H^2(\OG_\phi,k)$ of the Abelian extension $1\rightarrow \brak\phi \rightarrow \Gamma_\phi \rightarrow    \overline{\Gamma}_\phi\rightarrow 1$. Unlike in the finite order case, the homotopy equivalence is not an equivalence of cyclic complexes. Therefore, in order to get quasi-isomorphisms for the cyclic complex of $\Tot(C^\phi(\Gamma_\phi, \cA))$ we need to go beyond the scope of cylindrical modules. The natural setup is the setup of $S$-modules in the sense of Jones-Kassel~\cite{JK:KT89}, or more generally para-$S$-modules. In Section~\ref{sec:triangular-S-module} we introduce a ``cylindrical version" of $S$-modules, which we coin triangular $S$-modules. Examples of triangular $S$-modules are given by tensor products of some $S$-modules and mixed complexes. Moreover, any cylindrical complex $C$ gives rise to two triangular $S$-modules whose total $S$-modules give the cyclic complex of $\Tot(C)$ (see Section~\ref{sec:triangular-S-module}). In addition, triangular $S$-module lead us to elementary derivations of \emph{all} the spectral sequences considered in this note.   
 
We also introduce the notion of a good infinite order action (see Definition~\ref{eq:infinite.good-action}). We then obtain explicit quasi-isomorphisms in the case of good infinite order actions; under these  quasi-isomorphisms the periodicity operator is given by some  cap product with the Euler class $e_\phi$ (see Theorem~\ref{thm:infinite.HCAGphi}). For general infinite order elements, we also obtain a simple derivation of a spectral sequence which specializes to the spectral sequence of Feigin-Tsygan (Theorem~\ref{thm:infinite.HCAGphi-FT-sequence-general}). 

Using difficult homological algebra arguments, Nistor~\cite{Ni:InvM90} showed that, for general infinite order action, $\HC_\bt(\cA)_{[\phi]}$ is a module over the group cohomology of the normalizer $\OG_\phi$ and the periodicity operator is given by the action of the Euler class. We recover Nistor's results via a simple and general coproduct construction for paracyclic modules (see Theorem~\ref{thm:infinite.action-HC}). An improvement with respect to~\cite{Ni:InvM90} is the fact that the action is defined explicitly at the level of chains. Moreover, the relationship with the usual cap product becomes more transparent. 

In the second part of the paper (Sections~\ref{sec:LCA}--\ref{sec:varieties}), we deal with group actions on manifolds and varieties. In order to deal with group actions on manifolds we need to extend the results of first part to actions on locally convex algebras, where we replace the usual cyclic space $C(\cA_\Gamma)$ with its projective closure $\bC(\cA_\Gamma)$.  There is no major difficulty in carrying this out. In the same way as in the first part, the cyclic module of $\bC(\cA_\Gamma)$ splits into a direct sum of cyclic subspaces $\bC(\cA_\Gamma)_{[\phi]}$ parametrized by the conjugacy classes $[\phi]$ of elements $\phi \in \Gamma$. 
By a density argument, the results in the algebraic case can be extended to give a quasi-isomorphism between the cyclic complexes of $\bC(\cA_\Gamma)_{[\phi]}$ and a   mixed complex $\Tot(\bC^\phi(\Gamma_\phi, \cA))$ associated with the stabilizer $\Gamma_\phi$ of $\phi$. That mixed complex is the total mixed complex of a cylindrical space $\bC^\phi(\Gamma_\phi, \cA)$ obtained as the tensor product (over $\Gamma_\phi$) of the twisted  cyclic spaces $C^\phi(\Gamma_\phi)$ and $\bC^\phi(\cA)$ (see Section~\ref{sec:LCA}). Thereon the results of the first part apply \emph{verbatim} to the cylindrical space $\bC^\phi(\Gamma_\phi, \cA)$, and so they provide us with explicit quasi-isomorphisms for the cyclic complex of $\bC(\cA)$ (see Section~\ref{sec:LCA} for  the precise statements). 

In the context of an action of a general group $\Gamma$ on a manifold $M$, the relevant homology (resp., cohomology) theory is the equivariant homology (resp., cohomology) which is the  homology (resp., cohomology) of the homotopy quotient $E\Gamma \times_\Gamma M$ (a.k.a.~Borel construction). In particular, equivariant cohomology is the natural receptacle for the construction of equivariant characteristic classes. The equivariant homology/cohomology can also be defined simplicially in terms of the bicomplex of Bott~\cite{Bo:LNM78}. For our purpose, it is convenient to construct a ``mixed complex" version of equivariant homology. Regarding the group homology chain complex $C(\Gamma)$ and the de Rham complex of differential forms $\Omega^\bt(M)$ as mixed complexes, we may form their tensor product as a mixed bicomplex $C(\Gamma, M)$. The cyclic homology of the total mixed complex $\Tot(C(\Gamma,M))$ gives the mixed equivariant homology (see Section~\ref{sec:equivariant}). 

The results of the first part, and their versions for locally convex algebras in Section~\ref{sec:LCA}, are functorial in nature. In particular, any quasi-isomorphism for $\bC^\phi(\cA)$ can be input into this framework to give a quasi-isomorphism for $\bC(\cA_\Gamma)_{[\phi]}$. In the case of an action of $\Gamma$ on a manifold $M$, i.e., for $\cA=C^\infty(M)$, there is a twisted version of Connes-Hochschild-Kostant-Rosenberg theorem due to Brylinski~\cite{Br:AIF87} and Brylinski-Nistor~\cite{BN:KT94}. We then have an explicit quasi-isomorphism between $\bC^\phi(\cA)$ and the de Rham complex of the fixed-point set $M^\phi$. The cleanness condition ensures us that $M^\phi$ has a stratification by a disjoint union of submanifolds. 

The action is always clean when $\phi$ has finite order, including when $\Gamma_\phi$ is finite. When $\Gamma_\phi$ is finite we get an explicit quasi-isomorphism with the $\Gamma_\phi$-invariant de Rham complex of $M^\phi$ (Theorem~\ref{thm:LCA.bHCAGphi-finite}). As a consequence, when the action of $\Gamma$ is discrete and proper we recover the description of cyclic homology of $\cA_\Gamma$ of Baum-Connes~\cite{BC:CCDG} in terms of orbifold cohomology (\emph{cf}.~Corollary~\ref{cor:finite-Baum-Connes}). 

More generally, when $\phi$ has finite order we obtain explicit quasi-isomorphisms
with the mixed equivariant homology of $M^\phi$ (Theorem~\ref{thm:finite-order}). This refines previous results of Brylinski-Nistor~\cite{BN:KT94}. As a consequence we obtain an explicit construction of cyclic cycles over $\cA_\Gamma$ in terms of cap products of equivariant characteristic classes and group homology classes (see Corollary~\ref{cor:finite-order.cap-HG-HP}). This connects nicely with the work of  Connes-Moscovici~\cite{CM:CMP98}, Moscovici~\cite{Mo:AIM15} and Moscovici-Rangipour~\cite{MR:AIM11} on the construction of cyclic cocycles out of transverse characteristic classes through the channel of Hopf cyclic cohomology. 

We also calculate the pairing of the resulting cycles with cocycles arising from equivariant currents (Proposition~\ref{prop:manifolds.pairing-equivarian-currents}). Examples of such cocycles include the transverse fundamental class of Connes~\cite{Co:Kyoto83} and the CM cocycle of an equivariant Dirac spectral triple~(see \cite{PW:JNCG16}).
 In the infinite order case, we obtain explicit quasi-isomorphisms when the action is clean  (Theorem~\ref{thm:manifolds.infinite-order-clean}). This involves a version of equivariant cohomology associated with the normalizer $\OG_\phi =\Gamma_\phi/ \brak\phi$ and the de Rham complex of $M^\phi$ where the de Rham differential is combined with the Euler class of the extension of $\OG_\phi$ by $\Gamma_\phi$ (see Section~\ref{sec:manifolds}).  This fixes the misidentification of~\cite{Cr:KT99}. 
 
 The results for group actions on manifolds have complete analogues for group actions on smooth varieties thanks to the twisted version of Hochschild-Kostant-Rosenberg theorem in~\cite{BDN:AIM17}. When $\Gamma_\phi$ is finite we obtain an explicit quasi-isomorphism with the $\Gamma_\phi$-invariant de Rham complex of $X^\phi$ (Theorem~\ref{thm:varieties.finite}). In particular, we recover a recent result of Brodzki-Dave-Nistor~\cite{BDN:AIM17}. More generally, when $\phi$ has finite order we get an explicit quasi-isomorphism with a mixed bicomplex associated with the de Rham complex of $X^\phi$ (Theorem~\ref{thm:varieties.finite-order}). Finally, in the infinite order case, we also have an explicit quasi-isomorphism provided the action is clean (see Theorem~\ref{thm:varieties.infinite-order}). 

This paper is organized as follows. In Section~\ref{sec:background} we survey the main background on cyclic homology needed for this note. In Section~\ref{sec:triangular-S-module}, we introduce and study triangular $S$-modules. In Section~\ref{sec:CphiGA}, we construct the cylindrical complexes $C^\phi(\Gamma, \sC)$ whose roles are pivotal in this paper. In Section~\ref{sec:conj-classes}, we show that the computation of the cyclic homology of $\cA_\Gamma$ reduces to the study of those cylindrical complexes. In Section~\ref{sec:finite}, we look at the case where $\Gamma_\phi$ is finite. The finite order case is dealt with in Section~\ref{sec:finite-order}. In Section~\ref{sec:infinite-order}, we deal with the infinite order case. 
In Section~\ref{sec:LCA}, we explain how to extend the previous results to actions on locally convex algebras. In Section~\ref{sec:equivariant}, after recalling the main facts about equivariant cohomology, we introduce mixed equivariant homology and explain its relationship with equivariant cohomology. In Section~\ref{sec:manifolds} we apply the results of Section~\ref{sec:LCA} to crossed-product algebras associated with group actions on manifolds. 
In Section~\ref{sec:varieties} we obtain analogues of these results for group actions on smooth varieties. 
 
 The results of this paper were also announced in~\cite{Po:CRAS4, Po:CRAS5}. 

\begin{acknowledgements}
I wish to thank Paul Baum, Alain Connes, Sasha Gorokhovsky, Masoud Khalkhali, Henri Moscovici, Victor Nistor, Markus Pflaum, Hessel Posthuma, Bahram Rangipour, Xiang Tang, and Hang Wang for  various stimulating discussions related to the subject matter of this paper. I also would like to thank the hospitality of IH\'ES, McGill University, and UC Berkeley where part of the research for this paper was carried out.
\end{acknowledgements}

%\part{Algebraic Crossed-Products} 

\section{Background on Cyclic Homology}\label{sec:background}
In this section, we recall the main background and notation regarding cyclic homology. We refer to~\cite{Co:NCG,Lo:CH} and the references quoted below for more details. 
Throughout this section and the next two sections we let $k$ be an arbitrary unital ring. By a $k$-module we will always mean a left $k$-module. 

\subsection{Mixed complexes, cyclic modules, and $S$-maps}
A \emph{mixed complex} of $k$-modules is given by the datum of $(C_\bt,b,B)$, where $C_m$, $m\geq 0$, are $k$-modules and $b:C_\bt \rightarrow C_{\bt-1}$ and $B:C_\bt\rightarrow C_{\bt+1}$ are $k$-module maps such that $b^2=B^2=bB+Bb=0$ (\emph{cf}.~\cite{Bu:CM86,Ka:JAlg87}). The \emph{cyclic complex} of a mixed complex $C=(C_\bt,b,B)$ is the chain complex $C^\natural=(C^\natural_\bt, b+BS)$, where $C^\natural_m=C_m\oplus C_{m-2}\oplus \cdots$, $m\geq 0$, and $S:C_\bt^\natural \rightarrow C_{\bt-2}^\natural$ is the canonical projection obtained by factorizing out $C_m$. The operator $S$ is called the \emph{periodicity operator}. The \emph{periodic cyclic homology} is the $\Z_2$-graded chain complex $C^\sharp=(C^\sharp, b+B)$, where $C_i^\sharp = \prod_{q \geq 0}C_{2q+i}$, $i=0,1$. Equivalently, we have $C_\bt^\sharp=\varprojlim_S C_{2q+\bt}$, where $\varprojlim_S$ is the inverse limit of the directed system defined by the operators 
$S:C_{2q+\bt}\rightarrow C_{2q+\bt-2}$, $q\geq 1$. The homology of the chain complex $C^\natural$ (resp., $C^\sharp$) is called the \emph{cyclic homology} (resp., \emph{periodic cyclic homology}) of the mixed complex $C$. It is denoted by $\HC_\bt(C)$ (resp., $\HP_\bt(C)$). The homology of the chain complex $(C_\bt, b)$ is called the \emph{ordinary homology} of $C$ and is denoted by $H_\bt(C)$. 

A \emph{cyclic $k$-module} is given by the datum of $(C_\bt,d,s,t)$, where $C_m$, $m\geq 0$, are $k$-modules and the $k$-module maps $d:C_\bt\rightarrow C_{\bt-1}$, $s:C_\bt\rightarrow C_{\bt+1}$ and $t:C_\bt \rightarrow C_\bt$ define a simplicial module structure with faces $d_j=t^{j}dt^{-(j+1)}$ and degeneracies $s_j=t^{j+1}dt^{-(j+1)}$ and $t$ is further required to be cyclic, i.e., $t^{m+1}=1$ on $C_m$ (\emph{cf}.~\cite{Co:CRAS83,Co:NCG}). Any cyclic $k$-module $C=(C_\bt, d,s,t)$ gives rise to a mixed complex $(C_\bt, b,B)$, where $b=\sum_{j=0}^m(-1)^j d_j$ and $B=(1-\tau)sN$ on $C_m$ with $\tau =(-1)^m t$ and $N=1+\tau +\cdots + \tau^m$. With any unital algebra $\cA$ over $k$, Connes~\cite{Co:CRAS83,Co:NCDG,Co:NCG} associated a cyclic $k$-module $C(\cA)=(C_\bt(\cA), d,s,t)$, where $C_\bt(\cA)=\cA^{\otimes (m+1)}$, $m\geq 0$, and the structural operators 
$(d,s,t)$ are given by 
\begin{align}
 d(a^0\otimes \cdots \otimes a^m) &= (a^ma^0)\otimes a^1\otimes \cdots \otimes a^{m-1},\\
 s(a^0\otimes \cdots \otimes a^m) &= 1\otimes a^0\otimes \cdots \otimes a^m,
 \label{eq:cylic-algebra-s}\\ 
 t(a^0\otimes \cdots \otimes a^m) &= a^m\otimes a^0\otimes \cdots \otimes a^{m-1}, \qquad a^j \in \cA. 
\end{align}
The cyclic homology (resp., periodic cyclic homology) of $C(\cA)$ is simply called the \emph{ cyclic homology} (resp., \emph{ periodic cyclic homology}) of the algebra $\cA$. It is denoted by $\HC_\bt(\cA)$ (resp., $\HP_\bt(\cA)$) (\emph{cf}.~\cite{Co:Ober81,Co:NCDG,Co:NCG}).  

Let $C=(C_\bt, b,B)$ and $\wC=(\wC_\bt, b, B)$ be mixed complexes. Any mixed complex map $f:C_\bt \rightarrow \wC_\bt$ gives rise to chain maps  
$f:C_\bt^\natural \rightarrow \wC_\bt^\natural$ and  $f:C_\bt^\sharp \rightarrow \wC_\bt^\sharp$ between the corresponding cyclic and periodic complexes. We obtain quasi-isomorphisms whenever the chain map $f:(C_\bt, b)\rightarrow  (\wC_\bt, b)$ is a quasi-isomorphism. Moreover, according to~\cite{Ka:JAlg87} an \emph{$S$-map} $f:C^\natural_\bt \rightarrow \wC_\bt^\natural$ is a chain map which is compatible with the $S$-operators. It uniquely decomposes as $f=\sum f^{(j)} S^j$, where $f^{(j)}:C_\bt \rightarrow \wC_{\bt+2j}$ is a $k$-module of degree $2j$ such that $[b,f^{(0)}]=0$ and $[B,f^{(j)}]+[b,f^{(j+1)}]=0$. Any mixed complex map gives rise to an $S$-map. An $S$-map $f:C^\natural_\bt \rightarrow \wC_\bt^\natural$ is a quasi-isomorphism when its zeroth degree chain map $f^{(0)}:(C_\bt, b)\rightarrow (\wC,b)$ is a quasi-isomorphism. In addition, any $S$-map $f:C^\natural_\bt \rightarrow \wC_\bt^\natural$  extends to a chain map $f^\sharp: C^\sharp_\bt \rightarrow \wC_\bt^\sharp$. We obtain a quasi-isomorphism whenever $f$ is a quasi-isomorphism. 

\subsection{The Paracyclic and cylindrical categories}
In the context of group actions on algebras, we are naturally lead to go beyond the scope of mixed complexes and cyclic modules~(\cite{FT:LNM87,GJ:Crelle93}). According to~\cite{GJ:Crelle93}, a \emph{parachain complex} is given by the datum of $(C_\bt,b,B)$, where $C_m$, $m\geq 0$, are $k$-modules and $b:C_\bt \rightarrow C_{\bt-1}$ and $B:C_\bt\rightarrow C_{\bt+1}$ are $k$-module maps such that $b^2=B^2=0$ and $bB+Bb=1-T$, where $T:C_\bt\rightarrow C_\bt$ is some invertible $k$-module map. Given a parachain complex $C=(C_\bt,b+SB)$  we also can 
form a cyclic complex $C^\natural=(C^\natural_\bt, b,B)$ as in the case of mixed complexes. 
This need not be a chain complex (unless $C$ is a mixed complex), but we obtain a para-$S$-module (\emph{cf}.\ Section~\ref{sec:triangular-S-module}). 

Like a cyclic $k$-module, a \emph{paracyclic $k$-module} is given by the datum of $(C_\bt, d,s,t)$, where $C_m$, $m\geq 0$, are $k$-modules and the $k$-module maps $d:C_\bt\rightarrow C_{\bt-1}$, $s:C_\bt\rightarrow C_{\bt+1}$ and $t:C_\bt \rightarrow C_\bt$ define a simplicial module structure as above, but the cyclicity of the operator $t$ is replaced by the relation $t=ds$ and requiring $t$ to be invertible (\emph{cf}.~\cite{FT:LNM87,GJ:Crelle93}). Any paracyclic $k$-module gives rise to a parachain complex $(C_\bt, b,B)$, where $b$ is defined as above and $B=(1-\tau)s'N$ with $s'=sb's$ and $b'=\sum_{j=0}^{m-1} (-1)^j d_j=b-d$ on $C_m$. In that case $T=1-(bB+Bb)=t^{m+1}$ on $C_m$. When $C$ is a cyclic module we obtain a mixed complex which is isomorphic to the mixed complex defined above. 

A \emph{parachain bicomplex} is given by the datum of $(C_{\bt,\bt}, \wb, \wB, b,B)$, where $C_{p,q}$, $p,q\geq 0$, are $k$-modules, $(C_{\bt,q},\wb, \wB)$ and $(C_{p,\bt},b,B)$ are parachain complexes for all $p,q\geq 0$, and the horizontal differentials $(\wb, \wB)$ both commute with each of the vertical differentials $(b,B)$ (\emph{cf}.~\cite{GJ:Crelle93}). We obtain  a \emph{mixed bicomplex} when we further have $\wb\wB+\wB\wb=bB+Bb=0$. Following~\cite{GJ:Crelle93} we say that we have a \emph{cylindrical complex} when the operators $\wT=1-(\wb\wB+\wB\wb)$ and $T=1-bB+Bb$ are such that $\wT T=1$ (\emph{cf}.~\cite{GJ:Crelle93}). Any mixed bicomplex $C=(C_{\bt,\bt}, \wb,\wB, b,B)$ gives rise to a \emph{total mixed complex} $\Tot(C)=(\Tot_\bt(C),b^\dagger, B^\dagger)$, where $\Tot_m(C)=\oplus_{p+q=m} C_{p,q}$ and $b^\dagger=\wb +(-1)^pb$ and $B^\dagger = \wB +(-1)^pB$ on $C_{p,q}$. As observed in~\cite{GJ:Crelle93}, any cylindrical complex also gives rise to a total mixed complex $\Tot(C)=(\Tot_\bt(C),b^\dagger, B^\dagger)$, where $\Tot_\bt(C)$ and $b^\dagger$ are as above and $B^\dagger =\wB+(-1)^p \wT B$ on $C_{p,q}$.

A \emph{bi-paracyclic $k$-module} is given by the datum of $(C_{\bt,\bt}, \wbd, \ws, \wt, d,s,t)$, where $C_{p,q}$, $p,q\geq 0$, are $k$-modules, $(C_{\bt,q},\wbd, \ws, \wt)$ and $(C_{p,\bt},d,s,t)$ are paracyclic $k$-modules for all $p,q\geq 0$, and all the horizontal operators $(\wbd, \ws, \wt)$ commute with each of the vertical operators $(d,s,t)$ (\emph{cf}.~\cite{GJ:Crelle93}). We have a \emph{bi-cyclic $k$-module} when $\wt^{p+1}=t^{q+1}=1$ on $C_{p,q}$. We obtain \emph{cylindrical $k$-module} when $\wt^{p+1}t^{q+1}=1$ (\emph{cf}.~\cite{GJ:Crelle93}). Any bi-paracylic $k$-module $C=(C_{\bt,\bt}, \wbd, \ws, \wt, d,s,t)$ gives rise to a parachain bicomplex. We get a mixed bicomplex (resp., cylindrical complex) when $C$ is a bi-cyclic (resp., cylindrical) module. In addition, if $C=(C_{\bt,\bt}, \wbd, \ws, \wt, d,s,t)$ is a bi-paracyclic $k$-module, then its diagonal gives rise to a paracyclic $k$-module $\Diag(C)=(\Diag_\bt(C), \wbd d, \ws s, \wt t)$, where $\Diag_m(C)=C_{m,m}$, $m \geq 0$. We actually obtain a cyclic $k$-module when $C$ is a cylindrical module. 

\subsection{$S$-modules and perturbation lemma}
For the purpose of this paper, it is convenient to work in the general framework of  $S$-modules in the sense of Jones-Kassel~\cite{JK:KT89,Ka:Crelle90}. An \emph{$S$-module} is given by a chain complex of $k$-modules $(C_\bt, d)$ and a chain map $S:C_\bt \rightarrow C_{\bt-2}$ of degree $-2$,  which is called \emph{periodicity operator}. Given a mixed complex $C$, the associated cyclic complex $C^\natural$ is an $S$-module. When $C$ is a cyclic $k$-module, then its Connes' cyclic complex $(C^\lambda_\bt, b)$ and the total complex of its cyclic bicomplex (in the sense of Tsygan~\cite{Ts:UMN83} and Loday-Quillen~\cite{LQ:CMH84}) both give rise to $S$-modules. There is a natural notion of $S$-module maps. Given $S$-modules $C=(C_\bt, d,S)$ and  $\wC=(\wC_\bt, d,S)$, we shall say that two $S$-module maps $f:C_\bt\rightarrow \wC_\bt$ and $g:C_\bt \rightarrow \wC_\bt$ are \emph{$S$-homotopic} when there is an $S$-module map $h:C_\bt \rightarrow \wC_{\bt+1}$ such that $f-g=dh+hd$. This naturally leads us to notions of $S$-homotopy inverses of $S$-module maps and $S$-homotopy equivalences of $S$-modules. 

As with mixed complexes, for any $S$-module $C=(C_\bt, d,S)$, the $S$-operator gives rise to a periodic chain complex $C^\sharp=(C^\sharp_\bt, d)$, where $C_\bt =\varprojlim_S C_{2q+\bt}$. Any $S$-module map $f:C_\bt \rightarrow \wC_\bt$ between $S$-modules induce a chain map $C^\sharp_\bt \rightarrow  \wC^\sharp_\bt$. 
This chain map is a quasi-isomorphism (resp., homotopy equivalence), whenever $f$ is a quasi-isomorphism (resp., an $S$-homotopy equivalence). 

%According to~\cite{Po:para-S-modules}, 
A  \emph{para-$S$-module} is given by the datum of $(C_\bt, b, S)$, where $C_m$, $m\geq 0$, are $k$-modules and $d:C_\bt \rightarrow C_{\bt-1}$ and $S:C_\bt \rightarrow C_{\bt -2}$ are $k$-module maps commuting with each other such that $d^2=(1-T)S$, where $T:C_\bt\rightarrow C_\bt$ is a $k$-module map commuting with both $d$ and $S$. Given any parachain complex $C$, its cyclic pseudo-complex is an example of para-$S$-module. Notions of para-$S$-module maps and $S$-homotopies of para-$S$-module maps  make sense in the same way as in the context of $S$-modules. 
Therefore, although quasi-isomorphisms do not quite make sense for para-$S$-modules, $S$-homotopy inverse of a para-$S$-module map and $S$-homotopy equivalence of para-$S$-modules do make sense. In particular, we have a notion of $S$-homotopy inverse for any $S$-map of parachain complexes. 

Assume further that the category of left $k$-modules $\Mod(k)$ has a tensor product in the sense there is a bifunctor $\otimes_k: \Mod(k) \times \Mod(k) \rightarrow \Mod(k')$, where $k'$ is a sub-ring of $k$ containing $\Z$. When $k$ is commutative we shall take this tensor product to be the usual tensor product of $k$-modules, in which case $k'=k$. Given two parachain complexes of $k$-modules $C=(C_\bt,b,B)$ and $\sC=(\sC,b,B)$ we then can form their tensor product $C\otimes_k \sC$ as the parachain bicomplex of $k'$-modules $(C_\bt \otimes_k \sC_\bt, \wb, \wB, b,B)$, where we have denoted by $(\wb, \wB)$ (resp., $(b,B)$) the operators $(b\otimes 1, B\otimes 1)$ (resp., $(1\otimes b, 1\otimes B) $). If $C=(C_\bt, d,s,t)$ and $\sC=(\sC_\bt, d,s,t)$, then we similarly define their tensor product $C \otimes \sC$ as a bi-paracyclic $k'$-module 
$(C_\bt \otimes \sC_\bt, \wbd, \ws, \wt, d,s,t)$. 

The perturbation theory of Brown~\cite{Br:TwistedEZ} and Kassel~\cite{Ka:Crelle90} can be extended without much difficulties to the setting of para-$S$-modules. The details will appear in the forthcoming paper~\cite{Po:para-S-modules}. In what follows, we will only need the following result.

\begin{lemma}[\cite{Po:para-S-modules}]\label{lem:homotopy-lemma}
 Let $C=(C_\bt, b,B)$ and $\wC=(C_\bt, b,0)$ be parachain complexes of $k$-modules, where the $B$-differential of $\wC$ is zero. Let $f:C_\bt \rightarrow \wC_\bt$ be a parachain complex map, and assume that the induced chain map $f:(C_\bt, b)\rightarrow (\wC_\bt, b)$ has a homotopy inverse $g:\wC_\bt \rightarrow C_\bt$. Then
 \begin{enumerate}
\item The associated $S$-map $f:C^\natural_\bt \rightarrow \wC^\natural_\bt$ has an explicit $S$-homotopy inverse $S$-map $g^\natural:\wC^\natural \rightarrow C^\natural$ whose zeroth degree component is $g$.
 
\item For any parachain complex $\sC$, the $S$-maps $f\otimes 1: \Tot_\bt(C\otimes_k \sC)^\natural \rightarrow \Tot_\bt(\wC\otimes_k \sC)^\natural$ and 
$g\otimes 1: \Tot_\bt(\wC\otimes_k \sC)^\natural \rightarrow \Tot_\bt(C\otimes_k \sC)^\natural$ are $S$-homotopy inverses of each other. 
\end{enumerate}
 \end{lemma}

\begin{remark}
 When  $\wC$ has a non-zero $B$-differential $\wB$, a more general version of Lemma~\ref{lem:homotopy-lemma} holds at the expense of replacing $\wB$ by an auxiliary differential $\tilde{B}$, which need agree with $\wB$ (see~\cite{Ka:Crelle90, Po:para-S-modules}).
\end{remark}

\subsection{The Bi-paracyclic Eilenberg-Zilber theorem}
Given a bisimplicial $k$-module $C$, the Eilenberg-Zilber theorem asserts that the shuffle map $\shuffle: \Tot_\bt(C) \rightarrow \Diag_\bt(C)$ and the Alexander-Whitney map $\AW: \Diag_\bt(C)\rightarrow \Tot_\bt(C)$ are quasi-inverse chain maps. The first generalization of the Eilenberg-Zilber theorem to bi-cyclic modules is due to Hood-Jones~\cite{HJ:KT87}. This was further extended to cylindrical modules by Getzler-Jones~\cite{GJ:Crelle93} (see also~\cite{KR:CMB04, ZW:CMB14}). More precisely, given any cylindrical $k$-module $C$, the cyclic complexes $\Tot(C)^\natural$ and $\Diag(C)^\natural$ are quasi-isomorphic. 

We actually need a version of the Eilenberg-Zilber theorem for bi-paracyclic modules. In this context $\Tot^\bt(C)^\natural$ and $\Diag(C)^\natural$ need not be chain complexes anymore. However, they are para-$S$-modules, and so it makes sense to speak about an $S$-homotopy equivalence between them. Bauval~\cite{Ba:Preprint98} established such an equivalence when $C$ is a bi-cyclic module. More generally, by extending the approach of~\cite{Ba:Preprint98} to the bi-paracyclic setting we obtain the following result.

\begin{proposition}[\cite{Po:para-S-modules}]\label{prop:bi-paracyclic-EZ}
 Let $C$ be a bi-paracyclic $k$-module. Then there are explicit $S$-maps $\shuffle^\natural:  \Tot_\bt(C)^\natural \rightarrow \Diag_\bt(C)^\natural$ and 
 $\AW^\natural: \Diag_\bt(C)^\natural  \rightarrow  \Tot_\bt(C)^\natural$ which are $S$-homotopic inverses whose respective zeroth degree components are the 
 shuffle and Alexander-Whitney maps.  
\end{proposition}

\section{Triangular $S$-Modules and Spectral Sequences} \label{sec:triangular-S-module}
In this section, we introduce and study ``cylindrical versions" of para-$S$-modules.

\begin{definition}\label{def:horizontal-triangular-S-module}
 A \emph{(horizontal) triangular para-$S$-module} is given by the datum of $(C_{\bt,\bt},\wbd, b,B, S)$, where $C_{p,q}$, $p,q\geq 0$, are $k$-modules, 
 $(C_{\bt,q}, \wbd, S)$ is a para-$S$-module and $(C_{p,\bt}, b,B)$ is a parachain complex for all $p,q\geq 0$,  
 the horizontal operators $(d,S)$ both commute with each of the vertical differentials $(b,B)$. We say that such a datum is a \emph{triangular $S$-module} when  $\wbd^2+(bB+Bb)S=0$.
\end{definition}

Given a triangular para-$S$-module $C=(C_{\bt,\bt},\wbd, b,B, S)$ we can form its total para-$S$-module as $\Tot(C)=(\Tot_\bt(C), d^\dagger, S)$, where $\Tot_m(C)= \bigoplus_{p+q=m} C_{p,q}$ and $d^\dagger = \wbd + (-1)^p (b+BS)$ on $C_{p,q}$. When $C$ is a triangular $S$-module the condition $d^2+(bB+Bb)S=0$ ensures us that $(d^\dagger)^2=0$, and so we actually obtain an $S$-module.  

If $C=(C_{\bt,\bt}, \wb, \wB, b,B)$ is a parachain bicomplex, then $\Tot(C)^\natural$ is the total para-$S$-module of the triangular para-$S$-module  
$C^\wsigma: =(C^\wsigma_{\bt,\bt}, \wb+\wB S, b,B,S)$, where $C^\wsigma_{p,q}=C_{p,q}\oplus C_{p-2,q} \oplus \cdots$ and $S$ is the periodicity operator of $\Tot(C)^\natural$, which we regard as an operator $S:C_{\bt,\bt}^\wsigma\rightarrow C_{\bt-2,\bt}^\wsigma$. 

Triangular para-$S$-modules provide us with a natural framework for defining the tensor product of (para-)$S$-modules with mixed complexes and parachain complexes. Suppose we have a tensor product $\otimes_k: \Mod(k)\times \Mod(k)\rightarrow \Mod(k')$ as in Section~\ref{sec:background}. Let $C=(C_\bt, d, S)$ be a para-$S$-module of $k$-modules and $\sC=(\sC_\bt, b,B)$ a parachain complex of $k$-modules. We then form the tensor product $C\otimes_k \sC$ as the triangular para-$S$-module $(C^\natural_\bt \otimes_k \sC_\bt, d, b,B, S)$, where we have denoted by $(d,S)$ (resp., $(b,B)$) the operators $(d\otimes 1, S\otimes 1)$ (resp., $(1\otimes b, 1\otimes B)$).

There is a natural spectral sequence attached with any triangular $S$-module $C=(C_{\bt,\bt},\wbd, b,B, S)$. Let $\{F_p(C)\}_{p \geq 0}$ be the filtration by columns of $\Tot(C)$, so that $F_p(C)_{p+q}= C_{0,p+q}\oplus \cdots \oplus C_{p,q}$. This is a filtration of  chain complexes, and so it gives rise to a spectral sequence $\{E^r_{p,q}\}_{r\geq 0}$ converging to $H_\bt(\Tot(C))$. The $E^1$-term is $E^1_{p,q}=H_{p,q}^v(C)$, where $H_{p,q}^v(C)$ is the $(q+1)$-th homology $k$-module of 
the chain complex $(C_{p,\bt},b)$, $p\geq 0$. The relation $d^2+(bB+Bb)S=0$ implies that $\wbd$ induces a differential on $H_{\bt,q}^v(C)$ for all $q\geq 0$. We denote by $H_{\bt}^hH_q^v(C)$ the homology of $(H_{\bt,q}^v(C), \wbd)$, $q\geq 0$. This gives the $E^2$-term. Therefore, we arrive at the following statement.

\begin{proposition}\label{prop:triang-S-mod.horizontal-SC}
Let $C=(C_{\bt,\bt},\wbd, b,B, S)$ be a triangular $S$-module. Then the filtration by columns of $\Tot(C)$ gives rise to a spectral sequence $E^2_{p,q}(C)= H_{p}^hH_q^v(C) \Rightarrow H_{p+q}(\Tot(C))$. 
\end{proposition}

In Definition~\ref{def:horizontal-triangular-S-module} the roles of rows and columns are not symmetric. Interchanging their roles leads us to the following definition.

\begin{definition}
 A \emph{vertical triangular para-$S$-module} is given by the datum of $(C_{\bt,\bt},\wb, \wB, d, S)$, where $C_{p,q}$, $p,q\geq 0$, are $k$-modules, 
 $(C_{p,\bt}, d, S)$ is a para-$S$-module and $(C_{\bt,q}, \wb,\wB)$ is a parachain complex for all $p,q\geq 0$,  
 the vertical operators $(d,S)$ both commute with each of the horizontal differentials $(b,B)$. We say that such a datum is a \emph{vertical triangular $S$-module} when  
 $(\wb\wB+\wB\wb)S+d^2=0$. 
\end{definition}

Any vertical triangular para-$S$-module $C=(C_{\bt,\bt},\wb, \wB, d, S)$ gives rise to a  total para-$S$-module $\Tot(C)=
(\Tot_\bt(C), d^\dagger, S)$, where $d^\dagger =  \wb+\wB S+ (-1)^d$ on $C_{p,q}$. We obtain an $S$-module when $C$ is a triangular $S$-module. 
If $C=(C_{\bt,\bt}, \wb, \wB, b,B)$ is a parachain bicomplex, then $\Tot(C)^\natural$ is the total para-$S$-module of the vertical triangular para-$S$-module $C^\sigma=(C^\sigma_{\bt,\bt}, \wb, \wB, b+BS,S)$, where $C^\sigma_{p,q}=C_{p,q}\oplus C_{p,q-2} \oplus \cdots$ and $S$ is the periodicity operator of $\Tot(C)^\natural$, seen as an operator $S:C_{\bt,\bt}^\sigma\rightarrow C_{\bt,\bt-2}^\sigma$. 

Suppose that we have a tensor product on $\Mod(k)$ as above. Then given a parachain complex  $\sC=(\sC,b,B)$ and a para-$S$-module $C=(C_\bt, d,S)$,  we define the tensor product $\sC \otimes_k C$ as the vertical para-$S$-module $(\sC_\bt \otimes_k C, b, B, d, S)$. 

For any vertical triangular $S$-module $C=(C_{\bt,\bt},\wb, \wB, d, S)$, the filtration by rows of $\Tot(C)$ defines a spectral sequence $\{E^r_{p,q}(C)\}$ converging to its homology. For $p,q\geq 0$, let $H^h_{p,q}(C)$ be the $(p+1)$-th homology $k$-module of the chain complex $(C_{\bt,q},\wb)$. This gives the $E^1$-term. For all $p\geq 0$, we get a chain complex $(H^h_{p,\bt}(C),d)$ whose homology is denoted by $H_\bt^vH^h_p(C)$. This gives the $E^2$-term, and so we arrive at the following result.

\begin{proposition}
 Let $C$ be a vertical triangular $S$-module. Then the filtration by rows of $\Tot(C)$ gives rise to a spectral sequence 
 $E^2_{p,q}(C) = H_p^vH^h_q(C) \Rightarrow H_{p+q}(\Tot(C))$. 
\end{proposition}

In the case of a cylindrical $C$ we have a horizontal triangular $S$-module $C^\wsigma$ and a vertical $S$-module $\wC^\sigma$ whose total $S$-modules agree with the cyclic complex $\Tot(C)^\natural$. Therefore, we obtain two spectral sequences $\{E^r_{p,q}(C^\wsigma)\}$ and $\{E^r_{p.q}(C^\sigma)\}$ that converge to $\HC_\bt(\Tot(C))$. For $p,q\geq 0$ let $H^v_{p,q}(C)$ be the $(q+1)$-th homology $k$-module of the chain complex $(C_{p,\bt}, b)$. The differentials $\wb$ and $\wB$ descend to differentials on $H_{\bt,q}^v(C)$.  These differentials anticommute with each other, and so we obtain a mixed complex $H_q^v(C)=(H^v_{\bt,q},\wb,\wB)$. We then have $H^v_{p,q}(C^\wsigma)=H^v_{p,q}(C)$ and $H_{p}^hH_q^v(C)=\HC_p(H_q^v(C))$. We similarly define horizontal ordinary homology $k$-modules $H^h_{p,q}(C)$ and vertical mixed complexes $H^h_p(C)=(H_{p,\bt}, b,B)$, so that $H^h_{p,q}(C^\sigma)=H^h_{p,q}(C)$ and $H_{q}^vH_p^h(C^\sigma)= \HC_q(H_p^h(C))$. Therefore, we obtain the following result.

\begin{proposition}\label{prop:triang-S-mod.spectral-sequences-cyclindrical}
 Let $C$ be a cylindrical complex. Then the filtration by columns of $\Tot(C^\wsigma)$ and the filtration by rows of $\Tot(C^\sigma)$ give rise to spectral sequences, 
\begin{equation*}
E^2_{p,q}(C^\wsigma)= \HC_p(H_q^v(C)) \Rightarrow \HC_{p+q}(\Tot(C)), \qquad E^2_{p,q}(C^\sigma) = \HC_p(H_q^h(C)) \Rightarrow \HC_{p+q}(\Tot(C)). 
\end{equation*}
\end{proposition}

\begin{remark}
The spectral sequence of Getzler-Jones~\cite{GJ:Crelle93} is a special case of spectral sequence $E^2_{p,q}(C^\sigma)$.
\end{remark}

\emph{From now on, we assume that $k$ is a commutative ring containing $\Q$. In addition, we let $\Gamma$ be a group acting on a unital $k$-algebra $\cA$.}

\section{The Cylindrical Complexes $C^\phi(\Gamma,\sC)$} \label{sec:CphiGA}
As we shall show in the next section, the computation of the cyclic homology of the crossed-product algebra $\cA\rtimes \Gamma$ reduces to the cyclic homology of some cylindrical complexes $C^\phi(\Gamma,\sC)$. The goal of this section is to present the construction of those cylindrical complexes. 

The standard cyclic $k\Gamma$-module of $\Gamma$ is $C(\Gamma)=(C_\bt(\Gamma), d,s,t)$, where 
$C_m(\Gamma)=k \Gamma^{m+1}$, $m\geq 0$, and  the structural operators $(d,s,t)$ are given by
\begin{align}
d(\psi_0,\ldots,\psi_m)&= (\psi_0,\ldots, \psi_{m-1}),
\label{eq:cyclic-module-Gamma-d} \\
 s(\psi_0,\ldots,\psi_m)&= (\psi_m,\psi_0,\ldots, \psi_m),
\\
 t(\psi_0,\ldots,\psi_m)&= (\psi_m,\psi_0,\ldots, \psi_{m-1}), \qquad \psi_j\in \Gamma. 
\end{align}
Its geometric realization is the universal contractible $\Gamma$-bundle $E\Gamma$ (a.k.a.~Milnor's construction). Its $b$-differential is the standard group differential,  
\begin{equation}
 \partial(\psi_0,\ldots,\psi_m) =\sum_{0\leq j \leq m}(-1)^j (\psi_0,\ldots,\hat{\psi}_j, \ldots, \psi_m), \qquad \psi_j\in \Gamma. 
 \label{eq:cyclic-module-Gamma-partial}
\end{equation}
Given any $k\Gamma$-module $\sM$, the \emph{group homology} $H_\bt(\Gamma, \sM)$ is the homology of the chain complex $(C_\bt(\Gamma, \sM), \partial)$, where $C_m(\Gamma, \sM)=C_m(\Gamma)\otimes_\Gamma \sM$, $m \geq 0$. The \emph{group cohomology} $H^\bt(\Gamma, \sM)$ is the cohomology of the dual cochain complex $(C^\bt(\Gamma, \sM), \partial)$, where $C^m(\Gamma, \sM)$ consists of all $\Gamma$-equivariant maps $u:\Gamma^{m+1}\rightarrow \sM$ and $\partial u =u\circ \partial$. When $\sM=k$ we recover the homology and cohomology of the classifying space $B\Gamma=E\Gamma/\Gamma$ with coefficients in $k$. 

Let $\phi$ be a central element of $\Gamma$. We can use $\phi$ to twist the cyclic module structure of $C_\bt(\Gamma)$ to get the paracyclic $k\Gamma$-module $C^\phi(\Gamma)=(C_\bt(\Gamma),d,s_\phi,t_\phi)$, where the end-face $d$ is given by~(\ref{eq:cyclic-module-Gamma-d}) and the operators $(s_\phi, t_\phi)$ are given by
\begin{align*}
 s_\phi(\psi_0,\ldots,\psi_m)&= (\phi^{-1}\psi_m,\psi_0,\ldots, \psi_m), \\ 
 t_\phi(\psi_0,\ldots,\psi_m)&= (\phi^{-1}\psi_m,\psi_0,\ldots, \psi_{m-1}), \qquad \psi_j\in \Gamma. 
\end{align*}
For $\phi=1$ we recover the cyclic $k\Gamma$-module $C(\Gamma)$. The simplicial module structures of $C^\phi(\Gamma)$ and $C(\Gamma)$ agree, and so the $b$-differential of $C^\phi(\Gamma)$ is the group differential $\partial$ given by~(\ref{eq:cyclic-module-Gamma-partial}). Note also that $t_\phi^{m+1}(\psi_0,\ldots,\psi_m)=(\phi^{-1}\psi_0,\ldots,\phi^{-1}\psi_m)$, i.e., $t_\phi^{m+1}$ agrees with the action of $\phi^{-1}$ on $C_{m}(\Gamma)$. 
 
 In what follows,  we shall call \emph{ $\phi$-parachain complex} (of $k\Gamma$-modules) any parachain complex of $k\Gamma$-modules $\sC=(\sC_\bt,b,B)$ such that the operator $T=1-(bB+Bb)$ is given by the action of $\phi^{-1}$. We also define a \emph{ $\phi$-paracyclic $k\Gamma$-module} as a paracyclic $k\Gamma$-module $\sC=(\sC_\bt, d,s, t)$ such that $t^{m+1}$ is given by the action of $\phi^{-1}$. That is, the associated parachain complex is a $\phi$-parachain complex. (A $\phi$-paracylic $k$-module is also a $\theta$-cyclic object in the sense of Crainic~\cite{Cr:KT99}.) When $\phi=1$ a parachain complex (resp., $\phi$-paracyclic module) is just a usual mixex complex (resp., cyclic module). 
 
 The paracyclic $k\Gamma$-module $C^\phi(\Gamma)$ is a $\phi$-paracyclic module. Another example of $\phi$-paracyclic $k\Gamma$-module is the twisted cyclic $k\Gamma$-module  $C^\phi(\cA)= (C_\bt(\cA), d_\phi,s,t_\phi)$, where the extra degeneracy $s$ is given by~(\ref{eq:cylic-algebra-s}) and  the operators $(d_\phi, t_\phi)$ are given by 
\begin{align*}
 d_\phi(a^0\otimes \cdots \otimes a^m)&=(\phi^{-1}a^m)\otimes a^0\otimes \cdots \otimes a^{m}, \\
  t_\phi(a^0\otimes \cdots \otimes a^m)&=(\phi^{-1}a^m)\otimes a^0\otimes \cdots \otimes a^{m-1}, \qquad a^j\in \cA. 
\end{align*}
 
Given (left) $k\Gamma$-modules $\sM_1$ and $\sM_2$ we shall denote by $\sM_1\otimes_\Gamma \sM_2$ their tensor product over $k\Gamma$, i.e., the quotient of $\sM_1\otimes_k \sM_2$ by the action of $\Gamma$.  This provides us a with a bifunctor from $\Mod(k\Gamma)\times \Mod(k\Gamma)$ to $\Mod(k)$. Therefore,  if  $\sC$ and $\sC'$ are parachain complexes of $k\Gamma$-modules (resp., paracyclic $k\Gamma$-modules), then we can form their tensor product $\sC \otimes_\Gamma \sC'$ as in Section~\ref{sec:background} to get a parachain bicomplex of $k$-modules (resp., a bi-paracyclic $k$-modules). We observe that when $\sC$ and $\sC'$ are 
$\phi$-parachain complexes (resp., $\phi$-paracyclic modules), the tensor product $\sC \otimes_\Gamma \sC'$ is a cylindrical complex (resp., a cylindrical $k$-module).

In the following, when $\sC$ is  a $\phi$-parachain complex of $k\Gamma$-module (resp., $\phi$-paracyclic $k\Gamma$-module) 
we shall denote by $C^\phi(\Gamma, \sC)$ the cylindrical complex (resp., cylindrical $k$-module) $C^\phi(\Gamma)\otimes_\Gamma \sC$. We shall use the notation $C^\phi(\Gamma, \cA)$ for $\sC=C^\phi(\cA)$.

\begin{remark}
The diagonal cyclic $k$-module $\Diag(C^\phi(\Gamma, \cA))$ is isomorphic to the cyclic $k$-module $\cA_\phi[\Gamma]^{\natural \natural}$ of Feigin-Tsygan~\cite[\S\S 4.1]{FT:LNM87}. It is also isomorphic to the cyclic $k$-module $\tilde{L}(A,\Gamma, \phi)$ of Nistor~\cite[\S\S 2.5]{Ni:InvM90}. 
\end{remark}

\begin{remark}\label{rmk:CphiGA-Getzler-Jones}
The cylindrical module $C^\phi(\Gamma, \cA)$ itself is isomorphic to the cylindrical $k$-module $\cA_\phi^\natural$ of Getzler-Jones~\cite[\S 4]{GJ:Crelle93}. 
The main difference between the cylindrical modules $C^\phi(\Gamma, \cA)$ and $\cA_\phi^\natural$ lies in the fact that the former is realized as a tensor product over $\Gamma$ of two paracyclic $k\Gamma$-modules. From a geometric point of view, this is the natural simplicial counterpart to Borel's construction of the homotopy quotient $E\Gamma \times_\Gamma M$. This makes the computations be somewhat more transparent. 
\end{remark}

Finally, it can be shown that  the tensor products $C_m(\Gamma)\otimes_\Gamma -$, $m\geq 0$,  are exact functors from $\Mod (k \Gamma)$ to $\Mod(k)$. Therefore,  we obtain the following functoriality result.

\begin{proposition}\label{prop:functoriality-CphiAsC}
 Suppose that $\alpha: \sC_\bt \rightarrow \sC_\bt'$ is a quasi-isomorphism of $\phi$-parachain complexes. Then the mixed complex map
 $\alpha \otimes 1: \Tot_\bt(C^\phi(\Gamma, \sC))\rightarrow   \Tot_\bt(C^\phi(\Gamma, \sC'))$ is a quasi-isomorphism. 
\end{proposition}

\section{Splitting along Conjugacy Classes.}\label{sec:conj-classes} 
In this section, we explain how the computation of the cyclic homology of crossed-product algebras can be reduced to studying the cyclic structure of 
cylindrical modules of the type constructed in the previous section. 

In what follows we denote by $\cA_\Gamma$ the crossed-product algebra $\cA\rtimes \Gamma$. Recall that $\cA_\Gamma$ is a unital $k$-algebra with generators $a\in \cA$ and $u_\phi$, $\phi \in \Gamma$, subject to the relations,
\[
 a^0u_{\phi_0} a^1u_{\phi_1} =a^0(\phi_0^{-1} a^1)u_{\phi_0\phi_1}, \qquad a^j \in \cA, \ \phi_j\in \Gamma. 
 \] 

Given any $\phi \in \Gamma$, we shall denote by $[\phi]$ its conjugacy class in $\Gamma$. 
For $m\geq 0$, we also let $C_m(\cA_\Gamma)_{[\phi]}$ be the submodule of $C_m(\cA_\Gamma)$ generated by elementary tensor products $a^0u_{\phi_0}\otimes \cdots \otimes a^mu_{\phi_m}$, with $a^0,\ldots, a^m$ in $\cA$ and $\phi_0,\ldots, \phi_m$ in $\Gamma$ such that $\phi_0 \cdots \phi_m\in [\phi]$. This condition is invariant under the action of the structural operators $(d,s,t)$ of the cyclic $k$-module $C(\cA_\Gamma)$. Therefore, we obtain a cyclic submodule $C(\cA_\Gamma)_{[\phi]}$. We then have the following decomposition of cyclic $k$-modules, 
\begin{equation}
 C_\bt(\cA_\Gamma)= \bigoplus C_\bt(\cA_\Gamma)_{[\phi]},
 \label{eq:splitting-CAG}
\end{equation}
where the sum goes over all conjugacy classes in $\Gamma$. This provides us with a corresponding decomposition of cyclic complexes along with a canonical inclusion of periodic cyclic complexes. Let us denote by $\HC_\bt(\cA_\Gamma)_{[\phi]}$ (resp., $\HP_\bt(\cA_\Gamma)_{[\phi]}$) the cyclic homology (resp., periodic cyclic homology)
of $C(\cA_\Gamma)_{[\phi]}$. Then we have the following splitting and canonical inclusion of graded $k$-modules, 
\begin{equation}
 \HC_\bt(\cA_\Gamma)= \bigoplus \HC_\bt(\cA_\Gamma)_{[\phi]} \qquad  \text{and} \qquad  \bigoplus \HP_\bt(\cA_\Gamma)_{[\phi]}\subset  \HP_\bt(\cA_\Gamma).
 \label{eq:splitting-HCHP}
 \end{equation}
Moreover, the inclusion above is actually onto when $\Gamma$ has a finite number of conjugacy classes. 

Let $\phi \in \Gamma$ and denote by $\Gamma_\phi$ its centralizer in $\Gamma$. As $\phi$ is a central element of $\Gamma_\phi$, we may form the cylindrical complex $C^\phi(\Gamma_\phi,\cA)$ as in Section~\ref{sec:CphiGA}. We have a natural embedding of cyclic $k$-modules 
$\mu_\phi: \Diag_\bt(C^\phi(\Gamma_\phi,\cA))\rightarrow C_\bt(\cA_\Gamma)_{[\phi]}$ given by
 \begin{multline}
 \mu_{\phi}\left((\psi_{0},\ldots,\psi_{m})\otimes_{\Gamma_\phi}(a^{0}\otimes \cdots \otimes a^{m})\right)= \\ 
    [(\psi_{m}^{-1}\phi)\cdot a^{0}]u_{\phi\psi_{m}^{-1} \psi_{0}}\otimes 
     ( \psi_{0}^{-1}\cdot a^{1})u_{\psi_{0}^{-1}\psi_{1}}\otimes \cdots \otimes ( \psi_{m-1}^{-1}\cdot a^{m})u_{\psi_{m-1}^{-1}\psi_{m}}. 
     \label{eq:splitting.mu-phi}
 \end{multline}
 This embedding can be shown to be a quasi-isomorphism by using Shapiro's lemma. Furthermore, when $\phi=1$ the embedding $\mu=\mu_1$ is actually is an isomorphism of cyclic $k$-modules. Its inverse $ \mu^{-1}:C_\bt(\cA_\Gamma)_{[1]}\rightarrow \Diag_\bt(C(\Gamma,\cA))$ is such that, for all  $a^0,\ldots, a^m$ in $\cA$ and $\phi_0, \ldots, \phi_m$ in $\Gamma$ with $\phi_0 \cdots \phi_m=1$, we have
 \begin{equation}
 \mu^{-1}\left( a^0u_{\phi_0}\otimes \cdots \otimes a^mu_{\phi_m}\right)= \left( \hat{\phi}_0, \ldots, \hat{\phi}_{m}\right)\otimes_{\Gamma} \left( 
 a^0\otimes ( \hat{\phi}_0\cdot a^1) \otimes \cdots \otimes  ( \hat{\phi}_{m-1}\cdot a^m)\right), 
   \label{eq:splitting.inverse-mu}
\end{equation}
 where we have set $ \hat{\phi}_j=\phi_0 \cdots \phi_j$, $j=0,\ldots,m-1$.  
 
 Combining all this with the bi-paracyclic Eilenberg-Zilber theorem (Proposition~\ref{prop:bi-paracyclic-EZ}) we then arrive at the following statement.

\begin{proposition}\label{prop:quasi-isomorphism-CphiGA-CAGphi}
 Let $\phi \in \Gamma$. Then the following are quasi-isomorphisms,
 \begin{equation}
\Tot_\bt\left(C^\phi(\Gamma_\phi,\cA)\right)^\natural\xrightleftharpoons[\AW^\natural]{\shuffle^\natural} \Diag_\bt\left(C^\phi(\Gamma_\phi,\cA)\right)^\natural \xrightarrow{\mu_\phi} C_\bt(\cA_\Gamma)_{[\phi]}^\natural. 
\label{eq:quasi-isomorphism-CphiGA-CAGphi}
\end{equation}
There are similar results  at the level of the corresponding periodic cyclic complexes. 
\end{proposition}

The splitting and inclusion in~(\ref{eq:splitting-HCHP}) reduce the computation of the cyclic homology and periodic cyclic homology of $\cA_\Gamma$ to the study of the cyclic $k$-modules $C(\cA_\Gamma)_{[\phi]}$. Combining Proposition~\ref{prop:quasi-isomorphism-CphiGA-CAGphi} 
with Lemma~\ref{lem:homotopy-lemma} and Proposition~\ref{prop:functoriality-CphiAsC} further shows that the study of $C(\cA_\Gamma)_{[\phi]}$ 
can be split into studying separately each of the $\phi$-paracyclic $k\Gamma_\phi$-modules  $C^\phi(\Gamma_\phi)$ and $C^\phi(\cA)$. 

\section{The Cyclic Module $C(\cA_\Gamma)_{[\phi]}$ (Finite Centralizer Case)}\label{sec:finite}
Suppose that $\Gamma$ is finite and $\phi$ is a central element of $\Gamma$. Let $C(k)$ be the trivial complex $(C_\bt(k),0,0)$, where $C_0(k)=k$ and $C_m(k)=\{0\}$ for $m\geq 1$. We observe that, given any $\phi$-parachain complex $\sC=(\sC_\bt, b,B)$ the mixed complex $\Tot_\bt(C(k)\otimes_\Gamma \sC)$ is canonically identified with the $\Gamma$-invariant subcomplex $\sC^\Gamma=(\sC^\Gamma_\bt, b,B)$, where $\sC^\Gamma_m$ is the $\Gamma$-invariant submodule of $\sC_m$. Note also that $\sC^\Gamma$ is naturally a mixed complex, since $bB+Bb=1-\phi^{-1}=0$ on $\sC_\bt^\Gamma$. 

Let $\pi_0:C_\bt^\phi(\Gamma) \rightarrow C_\bt(k)$ be the $k\Gamma$-module map such that $\pi_0=0$ in degree~$m\geq 1$ and $\pi_0(\psi)=1$ for all $\psi \in \Gamma$ in degree~0. This is a map of parachain complexes. Moreover, the natural inclusion $\iota_\Gamma:C_\bt(k)\rightarrow C_\bt(\Gamma)$ is a right-inverse and homotopy left-inverse of the chain map $\pi_0:(C_\bt^\phi(\Gamma),\partial) \rightarrow (C_\bt(k),0)$. Therefore, by using Lemma~\ref{lem:homotopy-lemma} 
we obtain the following result. 

\begin{proposition}\label{prop:quasi-isom-Gfinite}
 Suppose that $\Gamma$ is finite and $\phi$ is a central element of $\Gamma$. 
 \begin{enumerate}
\item[(1)] The $S$-map $\pi_0: C_\bt^\phi(\Gamma)^\natural \rightarrow C_\bt(k)^\natural$ has an explicit right-inverse and $S$-homotopy left-inverse $\iota_\Gamma^\natural:  C_\bt(k)^\natural  \rightarrow C_\bt^\phi(\Gamma)^\natural$, which is an $S$-map whose zeroth degree component is $\iota_\Gamma$. 

\item[(2)] For any $\phi$-parachain complex $\sC$, the $S$-map $\iota_\Gamma^\natural: \sC^{\Gamma,\natural}_\bt \rightarrow \Tot_\bt(C^\phi(\Gamma,\sC))^\natural$ is a right-inverse and $S$-homotopy left-inverse of $\pi_0 \otimes 1:  \Tot_\bt(C^\phi(\Gamma,\sC))^\natural \rightarrow \sC^{\Gamma,\natural}_\bt$. 
\end{enumerate}
\end{proposition}

In particular, the 2nd part of Proposition~\ref{prop:quasi-isom-Gfinite} implements an explicit deformation retract of 
$\Tot(C^\phi(\Gamma,\sC))^\natural$ onto $\sC^{\Gamma,\natural}$. Combining this with Proposition~\ref{prop:functoriality-CphiAsC} 
and Proposition~\ref{prop:quasi-isomorphism-CphiGA-CAGphi} 
 we arrive at the following statement.

\begin{theorem}\label{thm:finite.quasi-isomorphism-CAGphi} 
Let $\phi\in \Gamma$ have a finite centralizer, and suppose we are given a quasi-isomorphism of $\phi$-parachain complexes $\alpha: C^\phi_\bt(\cA)\rightarrow \sC_\bt$.\begin{enumerate}
\item[(1)] We have explicit quasi-isomorphisms of cyclic complexes, 
\[
  \sC^{\Gamma_\phi,\natural}_\bt  \xleftarrow{\pi_0\otimes \alpha}
 \Tot_\bt\left(C^\phi(\Gamma_\phi,\cA)\right)^\natural\xrightleftharpoons[\AW^\natural]{\shuffle^\natural} \Diag_\bt\left(C^\phi(\Gamma_\phi,\cA)\right)^\natural \xrightarrow{\mu_\phi} C_\bt(\cA_\Gamma)_{[\phi]}^\natural.
 \] 
Moreover, if  $\alpha$ has a quasi-inverse (resp., homotopy inverse) $\beta:  \sC_\bt \rightarrow C^\phi_\bt(\cA)$, then 
$  (1\otimes \beta) \iota_{\Gamma_\phi}^\natural$ is a quasi-inverse (resp., homotopy inverse) of $\pi_0\otimes \alpha$.  

\item[(2)]  There are analogous statements for the corresponding periodic cyclic complexes. 

\item[(3)] This provides us with isomorphisms $\HC_\bt(\cA_\Gamma)_{[\phi]}\simeq \HC_\bt(\sC^{\Gamma_\phi})$ and $\HP_\bt(\cA_\Gamma)_{[\phi]}\simeq 
\HP_\bt(\sC^{\Gamma_\phi})$. 
\end{enumerate}
\end{theorem}

\begin{remark}
 The isomorphisms of the 3rd part are well known. For instance, they can be deduced from the spectral sequence of Feigin-Tsygan~\cite{FT:LNM87} (see also~\cite{GJ:Crelle93}). The novelty is that these isomorphisms follow from explicit quasi-isomorphisms. 
 \end{remark}

\begin{remark}
 The zeroth degree part of the $S$-map $(\pi_0\otimes \alpha) \AW^\natural$ is $(\pi_0\otimes \alpha) \AW$ and is given by
 \begin{equation*}
(\pi_0\otimes \alpha) \AW\left( (\psi_0,\ldots, \psi_m)\otimes_{\Gamma_\phi}(a^0\otimes \cdots \otimes a^m)\right) = 
\nu_{\Gamma_\phi}\alpha (a^0\otimes \cdots \otimes a^m),
\end{equation*}
where $\nu_{\Gamma_\phi}:\sC_\bt \rightarrow \sC^{\Gamma_\phi}$ is given by the average over the action of $\Gamma_\phi$.
The zeroth degree part of the $S$-map $\mu_\phi\shuffle^\natural\iota_{\Gamma_\phi}^\natural$ is equal to $\mu_\phi\shuffle\iota_{\Gamma_\phi}$, and so it is given by
\begin{equation*}
\mu_\phi\shuffle\iota_{\Gamma_\phi}(a^0\otimes \cdots \otimes a^m)= a^0u_\phi\otimes a^1 \otimes \cdots \otimes a^m, \qquad a^j \in \cA.
\end{equation*}
This provides us with a quasi-isomorphism from the twisted Hochschild complex of $C^\phi(\cA)^{\Gamma_\phi}$ to the Hochschild complex of $C(\cA_\Gamma)_{[\phi]}$ 
(see also~\cite{Da:JNCG13}).  
\end{remark}
 
\section{The Cyclic Module $C(\cA_\Gamma)_{[\phi]}$ (Finite Order Case)}\label{sec:finite-order} 
In this section, we look at the case where $\phi$ has finite order.  Let us first assume that $\phi$ is a central element of $\Gamma$ of order $r$, $r\in \N$. We shall relate $C^\phi(\Gamma)$ to the mixed complex 
$C^\flat(\Gamma)=(C(\Gamma),\partial,0)$. In what follows given a $\phi$-invariant mixed complex of $k\Gamma$-modules $\sC=(\sC_\bt, b,B)$, we shall denote by $C^\flat(\Gamma,\sC)$ the mixed bicomplex of $k$-modules $C^\flat(\Gamma)\otimes_\Gamma \sC$. More generally, if $\sC=(\sC_\bt, b,B)$ is a $\phi$-parachain complex, then we have a $\phi$-invariant subcomplex $\sC^\phi=(\sC_\bt^\phi, b,B)$, where $\sC^\phi_m$ is the $\phi$-invariant submodule of $\sC_m$. Note that $\sC^\phi$ is actually a mixed complex, since $bB+Bb=1-\phi^{-1}=0$ on $\sC^\phi_\bt$. This allows us to form the mixed bicomplex  $C^\flat(\Gamma,\sC^\phi)$. 

Bearing this in mind, let $\nu_\phi: C_\bt(\Gamma)\rightarrow C_\bt(\Gamma)$ be the $k\Gamma$-module map defined by
\begin{equation}
\nu_\phi(\psi_0,\ldots, \psi_m)= \frac{1}{r^{m+1}}\sum_{0\leq j \leq m} \sum_{0\leq \ell_j \leq r-1}(\phi^{\ell_0}\psi_0,\ldots, \phi^{\ell_m}\psi_m), \qquad \psi_j\in \Gamma. 
\end{equation}
In other words $\nu_\phi$ is the projection onto the $\phi$-invariant submodule $((k\Gamma)^\phi)^{\otimes (m+1)} \subset (k\Gamma)^{\otimes (m+1)}\simeq C_m(\Gamma)$. When $\phi=1$ this is just the identity map. 
We also let $\varepsilon : C_\bt(\Gamma) \rightarrow C_\bt(\Gamma)$ be the antisymmetrization map, 
\begin{equation}
\varepsilon(\psi_0,\ldots, \psi_m) = \frac{1}{(m+1)!} \sum_{\sigma \in \fS_m} (\psi_{\sigma^{-1}(0)}, \ldots, \psi_{\sigma^{-1}(m)}), \qquad \psi_j\in \Gamma,
\label{eq:antisymmetrization-map}
\end{equation}
 where $\fS_m$ is the group of permutations of $\{0,\ldots, m\}$. 
 
 As $\nu_\phi$ and $\varepsilon$ both are projections and commute with each other, the composition $\varepsilon \nu_\phi$ is a projection as well. It can be checked that $\varepsilon \nu_\phi:C_\bt^\phi(\Gamma) \rightarrow C^\flat(\Gamma)$ is a map of parachain complexes and the induced ordinary chain map $\varepsilon \nu_\phi:(C_\bt(\Gamma),\partial) \rightarrow (C_\bt(\Gamma),\partial)$ is chain homotopic to the identity (compare~\cite{Bu:CMH85, Ma:BCP86}). Since this is a projection, this means that we have a homotopy involution. In addition, we observe, that given any $\phi$-parachain complex, the map $(\varepsilon \nu_\phi)\otimes 1:
 C_{\bt}(\Gamma)\otimes_\Gamma \sC_\bt \rightarrow C_{\bt}(\Gamma)\otimes_\Gamma \sC_\bt $ maps to $C_\bt(\Gamma)\otimes \sC^\phi_\bt$, so that we obtain a parachain bicomplex map $(\varepsilon \nu_\phi)\otimes 1: C^\phi_{\bt,\bt}(\Gamma,\sC)\rightarrow C^\flat_{\bt,\bt}(\Gamma, \sC^\phi)$. Therefore, by using 
 Lemma~\ref{lem:homotopy-lemma} we 
 obtain the following result.

\begin{proposition}\label{prop:quasi-isomCphiGsC-finite-order}
Suppose that $\phi$ is a central element of $\Gamma$ of finite order. 
\begin{enumerate}
\item[(1)] The $S$-map $\varepsilon \nu_\phi:C_\bt^\phi(\Gamma)^\natural \rightarrow C^\flat_\bt(\Gamma)^\natural$ has an explicit $S$-homotopy inverse  
$(\varepsilon \nu_\phi)^\flat : C^\flat_\bt(\Gamma)^\natural  \rightarrow   C^\phi_\phi(\Gamma,\sC)^\natural$, which is an $S$-map whose zeroth degree component is 
$\varepsilon \nu_\phi$.  

\item[(2)] For any $\phi$-parachain complex $\sC$, the $S$-maps $(\varepsilon \nu_\phi)\otimes 1: \Tot_\bt(C^\phi(\Gamma,\sC))^\natural \rightarrow \Tot_\bt(C^\flat(\Gamma,\sC^\phi))^\natural$  and $(\varepsilon \nu_\phi)^\flat \otimes 1:  \Tot_\bt(C^\flat(\Gamma,\sC^\phi))^\natural  \rightarrow   \Tot_\bt(C^\phi(\Gamma,\sC))^\natural$ are $S$-homotopy inverses. 
\end{enumerate}
\end{proposition}

Combining Proposition~\ref{prop:quasi-isomCphiGsC-finite-order} with Proposition~\ref{prop:functoriality-CphiAsC} 
and Proposition~\ref{prop:quasi-isomorphism-CphiGA-CAGphi} gives the following result. 

\begin{theorem} \label{thm:finite-order.HCAGphi}
Let $\phi \in \Gamma$ have finite order, and suppose we are given a quasi-isomorphism of $\phi$-parachain complexes $\alpha: C^\phi_\bt(\cA)\rightarrow \sC_\bt$.  \begin{enumerate}
\item[(1)] We have explicit quasi-isomorphisms of cyclic complexes, 
\begin{equation}
  \Tot_\bt\left(C^\flat(\Gamma_\phi, \sC^\phi)\right)^\natural  \xleftarrow{(\varepsilon \nu_\phi)\otimes \alpha}
 \Tot_\bt\left(C^\phi(\Gamma_\phi,\cA)\right)^\natural\xrightleftharpoons[\AW^\natural]{\shuffle^\natural} \Diag_\bt\left(C^\phi(\Gamma_\phi,\cA)\right)^\natural \xrightarrow{\mu_\phi} C_\bt(\cA_\Gamma)_{[\phi]}^\natural.
 \label{eq:finite-order.quasi-isom-CAGphi}
 \end{equation}
Moreover, if $\alpha$ has a quasi-inverse (resp., homotopy inverse) $\beta:  \sC_\bt \rightarrow C^\phi_\bt(\cA)$, then 
$ (\varepsilon \nu_\phi)^\flat \otimes \beta$ is a quasi-inverse (resp., homotopy inverse) of $(\varepsilon \nu_\phi)\otimes \alpha$. 

\item[(2)] There are analogous statements for the corresponding periodic cyclic complexes. 

\item[(3)] This provides us with isomorphisms,
\begin{equation*}
\HC(\cA_\Gamma)_{[\phi]}\simeq \HC_\bt( \Tot(C^\flat(\Gamma_\phi,\sC^\phi))), \qquad \HP(\cA_\Gamma)_{[\phi]}\simeq \HP_\bt( \Tot(C^\flat(\Gamma_\phi,\sC^\phi))). 
\end{equation*}
\end{enumerate}
\end{theorem}

\begin{remark}
 When $\cA=k$ the crossed-product algebra $\cA_\Gamma$ is just the group ring $k\Gamma$. In this case, we recover the descriptions of $\HC_\bt(k\Gamma)_{[\phi]}$ and $\HP_\bt(k\Gamma)_{[\phi]}$ by Burghelea~\cite{Bu:CMH85}.
\end{remark}

As mentioned in Section~\ref{sec:triangular-S-module}, the mixed bicomplex $C^\flat(\Gamma_\phi, \sC^\phi)$ gives rise to two triangular $S$-modules. The horizontal triangular $S$-module is the tensor product $C^\flat(\Gamma_\phi)^\natural\otimes_{\Gamma_\phi} \sC^\phi$. The vertical triangular $S$-module is the tensor product  $C^\flat(\Gamma_\phi)\otimes_{\Gamma_\phi} \sC^{\phi,\natural}$. Here $C^\flat(\Gamma_\phi)^\natural$ and $\sC^{\phi,\natural}$ are regarded as $S$-modules and the tensor products are defined as in Section~\ref{sec:triangular-S-module}. The total $S$-modules of these triangular $S$-modules agree with 
$ \Tot\left(C^\flat(\Gamma_\phi, \sC^\phi)\right)^\natural$. Combining this with Theorem~\ref{thm:finite-order.HCAGphi} we see that the spectral sequences associated with these total $S$-modules give rise to spectral sequences converging to $\HC_\bt(\cA_\Gamma)_{[\phi]}$. More specifically, we have the following statement.

\begin{corollary}\label{cor:finite-order.HCAGphi-spectral-sequences}
 Suppose that the assumptions of Theorem~\ref{thm:finite-order.HCAGphi} hold. 
 \begin{enumerate}
\item[(1)] The filtrations by rows and columns of $\Tot_\bt(C^\flat(\Gamma_\phi)^\natural\otimes_{\Gamma_\phi} \sC^\phi)=  \Tot_\bt\left(C^\flat(\Gamma_\phi, \sC^\phi)\right)^\natural$ and the quasi-isomorphisms~(\ref{eq:finite-order.quasi-isom-CAGphi}) give rise to spectral sequences, 
\begin{gather}
  {~}^{I}\! E^{2}_{p,q}=\HC_p\left(H_q(\Gamma_\phi, \sC^\phi)\right) \Longrightarrow 
        \HC_{p+q}( \cA_{\Gamma})_{[\phi]},  
        \label{eq:spectral-sequence-finite-order1}
        \\
  {~}^{I\!I}\! E^{2}_{p,q}=H_p(\Gamma_\phi, \HC_q(\sC^\phi)) \Longrightarrow 
        \HC_{p+q}( \cA_{\Gamma})_{[\phi]}. 
        \label{eq:spectral-sequence-finite-order2}       
\end{gather}
Here $H_q(\Gamma_\phi, \sC^\phi)$ is the mixed complex $(H_q(\Gamma_\phi, \sC^\phi), b,B)$.

\item[(2)] The filtration by columns of $\Tot_\bt(C^\flat(\Gamma_\phi)^\natural\otimes_{\Gamma_\phi} \sC^\phi)=  \Tot_\bt\left(C^\flat(\Gamma_\phi, \sC^\phi)\right)^\natural$ and the quasi-isomorphisms~(\ref{eq:finite-order.quasi-isom-CAGphi}) give rise to a spectral sequence, 
        \begin{equation}
{~}^{I\!I\!I}\! E^{2}_{p,q} \Longrightarrow 
        \HC_{p+q}( \cA_{\Gamma})_{[\phi]}, \quad \text{where}\ {~}^{I}\! E^{2}_{p,q}=H_p(\Gamma_\phi, H_q(\sC^\phi))\oplus H_p(\Gamma_\phi, H_q(\sC^\phi)) \oplus \cdots,  
                \label{eq:spectral-sequence-finite-order3}
\end{equation}
and $H_\bt(\sC^\phi)$ is the ordinary homology of $\sC^\phi$. 
\end{enumerate}
\end{corollary}

\begin{remark}
 The spectral sequence~(\ref{eq:spectral-sequence-finite-order1}) is a refinement of the spectral sequence of Getzler-Jones~\cite[Theorem~4.5]{GJ:Crelle93}. 
\end{remark}

\begin{remark}
 As a chain complex $\Tot(C(\Gamma_\phi)\otimes_{\Gamma_\phi} \sC^{\phi,\natural})$ is the total complex of the double complex given by the tensor product of the chain complexes $(C(\Gamma_\phi)_\bt,\partial)$ and $\sC^{\phi, \natural}=(\sC_\bt^{\phi, \natural},b+BS)$. Therefore, its filtration by columns  too is a filtration of chain complexes. It gives rise to the spectral sequence~(\ref{eq:spectral-sequence-finite-order2}). When $\sC=C^\phi(\cA)$ and $\alpha=\op{id}$ we recover the spectral sequence of Feigin-Tsygan~\cite[Theorem~4.1.1]{FT:LNM87}.
\end{remark}

\begin{remark}
  The spectral sequence~(\ref{eq:spectral-sequence-finite-order3}) seems to be new. Note that its $E_2$-term does not involve the cyclic structure of 
  $C(\cA_\Gamma)_{[\phi]}$.
\end{remark}

\section{The Cyclic Module $C(\cA_\Gamma)_{[\phi]}$ (Infinite Order Case)}\label{sec:infinite-order}

In this section, we compute the cyclic homology of $C(\cA_\Gamma)_{[\phi]}$ when $\phi$ has infinite order. 

\subsection{From $C^\phi(\Gamma,\sC)$ to $C^\sigma(\OG,\sC)$} 
Suppose that $\phi$ is a central element of $\Gamma$ with infinite order. Set $\overline{\Gamma}=\Gamma / \brak\phi$, where $\brak\phi$ is the subgroup generated by $\phi$. The canonical projection $\wpi:\Gamma \rightarrow \overline{\Gamma}$ gives rise to a $(\Gamma, \overline{\Gamma})$-equivariant paracyclic module map $\wpi: C^\phi_\bt(\Gamma)\rightarrow C_\bt(\overline{\Gamma})$. 
Composing this map with the natural projection $\pi^\natural: C_\bt(\overline{\Gamma})^\natural\rightarrow C_\bt(\overline{\Gamma})$ and the antisymmetrization map $\varepsilon:C_\bt(\overline{\Gamma}) \rightarrow C_\bt(\overline{\Gamma})$ in~(\ref{eq:antisymmetrization-map}), we then obtain a $(\Gamma, \overline{\Gamma})$-equivariant chain map $\wepsn: C^\phi_\bt(\Gamma)^\natural \rightarrow C_\bt(\overline{\Gamma})$. If $\sM$ is a $\phi$-invariant $k\Gamma$-module, then the action of $\Gamma$ on $\sM$ descends to an action of $\overline{\Gamma}$, and so we obtain a chain map $\wepsn\otimes 1:  C_\bt^\phi(\Gamma,\sM)^\natural \rightarrow C_\bt(\overline{\Gamma},\sM)$. 

Let $C^\phi(\Gamma)^\lambda=(C^\phi_\bt(\Gamma)^\lambda, \partial)$ be the Connes complex of $C^\phi(\Gamma)$, where $C^\phi_\bt(\Gamma)^\lambda=C_\bt(\Gamma)/ \ran(1-\tau_\phi)$ with $\tau_\phi=(-1)^m t_\phi$ on $C_m(\Gamma)$. Note that $\phi$ acts like the identity on $C^\phi(\Gamma)^\lambda$, and so the action of $\Gamma$ descends to an action of $\overline{\Gamma}$.  The projection $\pi^\lambda: C^\phi_\bt(\Gamma)^\natural \rightarrow C^\phi_\bt(\Gamma)^\lambda$ is a chain map. Furthermore, the projection $\wpi: C_\bt(\Gamma)\rightarrow C_\bt(\OG)$ descends to a chain map $\wpi: C_\bt^\phi(\Gamma)^\lambda \rightarrow C_\bt(\overline{\Gamma})$ in such a way that $\wepsn=\varepsilon \wpi \pi^\lambda$.  

It was shown by Marciniak~\cite{Ma:BCP86} that the chain complex $C^\phi(\Gamma)^\lambda$ gives rise to a projective resolution of the trivial $k\overline{\Gamma}$-module $k$ (see also~\cite{Bu:CMH85}). It then follows that the chain map $\varepsilon \wpi$ has a homotopy inverse $\gamma: C_\bt(\overline{\Gamma})\rightarrow C^\phi_\bt(\Gamma)^\lambda$. Moreover, it follows from results of Kassel~\cite{Ka:Crelle90} that there is an explicit graded $k\Gamma$-module map $\iota: C_\bt(\Gamma) \rightarrow C_\bt(\Gamma)^\natural$ such that, for every $\phi$-invariant $k\Gamma$-module $\sM$, the $k$-module map $\iota\otimes 1:  C_\bt(\Gamma,\sM) \rightarrow C_\bt^\phi(\Gamma,\sM)^\natural$ descends to a chain map $(\iota \otimes 1)^\lambda:  C_\bt(\Gamma,\sM)^\lambda \rightarrow C_\bt^\phi(\Gamma,\sM)^\natural$ which is a homotopy inverse of $\pi^\lambda \otimes 1$. Therefore, we arrive at the following statement.

\begin{lemma}\label{lem:homotopy-inverses-sM}
 For any $\phi$-invariant $k\Gamma$-module $\sM$, the chain maps $\wepsn\otimes 1:  C_\bt^\phi(\Gamma,\sM)^\natural \rightarrow C_\bt(\overline{\Gamma},\sM)$ and 
 $(\iota \otimes 1)^\lambda(\gamma \otimes 1): C_\bt(\overline{\Gamma},\sM) \rightarrow C_\bt^\phi(\Gamma,\sM)^\natural $ are homotopy inverses. In particular, this provides us with an isomorphism $\HC_\bt(C^\phi(\Gamma, \sM))\simeq H_\bt(\overline{\Gamma},\sM)$. 
\end{lemma}

The periodicity operator $S:C^\phi_\bt(\Gamma)\rightarrow C^\phi_{\bt-2}(\Gamma)$ is closely related to the cap product with the Euler class $e_\phi \in H^2(\overline{\Gamma},k)$ of the Abelian extension $1\rightarrow \brak\phi \rightarrow \overline{\Gamma} \rightarrow \Gamma \rightarrow 1$ (\emph{cf}.\ \cite{Bu:CMH85}). Let $u_\phi \in C^2(\overline{\Gamma},k)$ be any $2$-cocycle representing $e_\phi$. The cap product with $u_\phi$ then gives rise to a chain map $u_\phi \frown -: H_{\bt}(\overline{\Gamma})\rightarrow H_{\bt-2}(\overline{\Gamma})$.

\begin{lemma}[compare~\cite{Ji:KT95}] \label{lem:infinite-order.homotopy-S}
There is an explicit $(\Gamma, \overline{\Gamma})$-equivariant map $h_\phi: C_\bt^\phi(\Gamma)^\natural \rightarrow C_{\bt -1}(\overline{\Gamma})$ such that 
$\wepsn S- (u_\phi \frown -) \wepsn = \partial h_\phi + h_\phi(\partial +B_\phi S)$. 
 \end{lemma}

It immediately follows from Lemma~\ref{lem:infinite-order.homotopy-S} that, for any $\phi$-invariant $k\Gamma$-module, under the isomorphism   
$\HC_\bt(C^\phi(\Gamma, \sM))\simeq H_\bt(\overline{\Gamma},\sM)$ provided by Lemma~\ref{lem:homotopy-inverses-sM}, the periodicity operator of $\HC_\bt(C^\phi (\Gamma,\sM))$ is given by
 the cap product with the Euler class $e_\phi \in H^2(\overline{\Gamma},k)$. 

As the cap product with $u_\phi$ is a chain map, we obtain an $S$-module $C^\sigma(\OG):=(C_\bt(\OG), \partial, u_\phi\frown -)$. 
Given any $\phi$-invariant mixed complex $\sC=(\sC_\bt,b,B)$, we shall denote by $C^\sigma(\OG,\sC)$ the triangular $S$-module given by the tensor product 
$C^\sigma(\OG)\otimes_{\OG} \sC$ as defined in Section~\ref{sec:triangular-S-module}. This gives rise to the total $S$-module
 $\Tot(C^\sigma(\overline{\Gamma},\sC))= (\Tot_\bt(C^\sigma(\overline{\Gamma},\sC)), d^\dagger, u_\phi \frown -)$, where
 \begin{equation}
\Tot_m\left(C^\sigma(\overline{\Gamma},\sC)\right)=\bigoplus_{p+q=m} C_p(\OG, \sC_q), \qquad d^\dagger =\partial +(-1)^pb +(-1)B(u_\phi \frown -). 
\label{eq:infinite.TotCsGsC}
\end{equation}
We then obtain a chain map $\theta:\Tot_\bt(C^\phi(\Gamma,\sC))^\natural \rightarrow \Tot_\bt(C^\sigma(\overline{\Gamma},\sC))$ by letting
\begin{equation}
\theta=\wepsn \otimes 1 + (-1)^{p-1} (1\otimes B)(h_\phi \otimes 1) \qquad \text{on $C_{p,q}(\Gamma, \sC)$}. 
\label{eq:infinite-order.theta}
\end{equation}
By using Lemma~\ref{lem:homotopy-inverses-sM}, Lemma~\ref{lem:infinite-order.homotopy-S}, and Proposition~\ref{prop:triang-S-mod.horizontal-SC} we then obtain the following result.

\begin{proposition}\label{prop:infinite-order.quasi-isom-sC}
For any $\phi$-invariant mixed complex $\sC$, the chain map~(\ref{eq:infinite-order.theta}) is a quasi-isomorphism. Moreover, we have
\begin{equation*}
\theta S - (u_\phi \frown -)S= d^\dagger (h_\phi \otimes 1) + (h_\phi \otimes 1) d^\dagger,
\end{equation*}
where we have also denoted by $d^\dagger$ the differential of $\Tot(C^\phi(\Gamma,\sC))^\natural$.
\end{proposition}

\subsection{Good infinite order actions}
Given a $\phi$-parachain complex $\sC=(\sC_\bt,b,B)$, we can form its $\phi$-coinvariant complex $\sC_\phi=(\sC_{\phi, \bt},b,B)$, where $\sC_{\phi, m}=\sC_m/\ran(\phi-1)$. Note that $\sC_\phi$ is a $\phi$-invariant mixed complex. Moreover, the canonical projection $\pi: \sC_\bt \rightarrow \sC_{\phi, \bt}$ is a parachain complex map.

\begin{definition}
  We shall say that a $\phi$-parachain complex $\sC$ is \emph{good} when the canonical projection $\pi: \sC_\bt \rightarrow \sC_{\phi, \bt}$ is a quasi-isomorphism.
\end{definition}

It can be shown that $\sC$ is good if and only if it is quasi-isomorphic to a $\phi$-invariant mixed complex. For instance, $\sC$ is good when we have a splitting $\sC_\bt =\sC^\phi_\bt \oplus \ran (\phi -1)$, where $\sC^\phi=\ker (\phi -1)$. In that case the inclusion of $\sC^\phi_\bt$ into $\sC_\bt$  is a quasi-isomorphism (see~\cite[Proposition~2.1]{HK:KT05}). We refer to~\cite[Example 3.10]{HK:JKT09} for an example of $\phi$-paracyclic module which does not give rise to a good $\phi$-parachain complex.

\begin{definition}\label{eq:infinite.good-action}
  Let $\phi \in \Gamma$. We shall say that the action of $\phi$ on $\cA$ is \emph{good} when $C^\phi(\cA)$ is a good $\phi$-parachain complex. 
\end{definition}

\begin{remark}
  We refer to Section~\ref{sec:manifolds} and Section~\ref{sec:varieties} for examples of good actions in the context of \emph{clean} actions on manifolds and smooth varieties.
 \end{remark}

Combining Proposition~\ref{prop:infinite-order.quasi-isom-sC} with Proposition~\ref{prop:functoriality-CphiAsC} and Proposition~\ref{prop:quasi-isomorphism-CphiGA-CAGphi} gives the following result. 

\begin{theorem}\label{thm:infinite.HCAGphi}
 Let $\phi \in \Gamma$ have infinite order, and set $\overline{\Gamma}_\phi = \Gamma_\phi/ \brak\phi$. Assume that the action of $\phi$ on $\cA$ is good, and let $\alpha: C^\phi_\bt(\cA)\rightarrow \sC_\bt$ be a parachain map and quasi-isomorphism, where $\sC$ is a $\phi$-invariant mixed complex. 
 \begin{enumerate}
\item[(1)] We have the following quasi-isomorphisms of chain complexes, 
  \begin{equation}
 \Tot_\bt (C^\sigma(\overline{\Gamma}_\phi ,\sC))  \xleftarrow{\theta(1\otimes \alpha)}
 \Tot_\bt\left(C^\phi(\Gamma_\phi,\cA)\right)^\natural\xrightleftharpoons[\AW^\natural]{\shuffle^\natural} \Diag_\bt\left(C^\phi(\Gamma_\phi,\cA)\right)^\natural \xrightarrow{\mu_\phi} C_\bt(\cA_\Gamma)_{[\phi]}^\natural.
 \label{eq:infinite.quasi-isom-CAG} 
\end{equation}
This provides us with an isomorphism,
\begin{equation}
\HC_\bt(\cA_\Gamma)_{[\phi]}\simeq H_\bt\left( \Tot(C^\sigma(\overline{\Gamma}_\phi ,\sC))\right).
\label{eq:infinite.isom-HCAG-HTotCsGsC} 
\end{equation}

\item[(2)] Under isomorphism~(\ref{eq:infinite.isom-HCAG-HTotCsGsC}) the periodicity operator of $\HC_\bt(\cA)_{[\phi]}$ is given by the cap product, 
\begin{equation*}
e_\phi \frown - : H_\bt(\Tot(C^\sigma(\overline{\Gamma}_\phi ,\sC))) \longrightarrow H_{\bt-2}(\Tot(C^\sigma(\overline{\Gamma}_\phi ,\sC))),
\end{equation*}
where $e_\phi\in H^2(\OG_\phi,k)$ is the Euler class of the extension 
$1\rightarrow \brak\phi \rightarrow \Gamma_\phi \rightarrow    \overline{\Gamma}_\phi\rightarrow 1$. 
\end{enumerate}
\end{theorem}

\begin{remark}
  When $\cA=k$ we recover the description of $\HC_\bt(k\Gamma)_{[\phi]}$ and $\HP_\bt(k \Gamma)_{[\phi]}$ by Burghelea~\cite{Bu:CMH85}.
\end{remark}

Combining Theorem~\ref{thm:infinite.HCAGphi} with Proposition~\ref{prop:triang-S-mod.horizontal-SC} we obtain the following corollary.

\begin{corollary}\label{cor:infinite.FT-spectral-sequence-good}
 Under the assumptions of Theorem~\ref{thm:infinite.HCAGphi}, the quasi-isomorphisms~(\ref{eq:infinite.quasi-isom-CAG}) and the filtration by columns of $\Tot_\bt(C^\sigma(\OG_\phi, \sC))$ give rise to a spectral sequence,
 \begin{equation}
E^2_{p,q}= H_p(\OG_\phi, H_q(\sC)) \Longrightarrow \HC_{p+q}(\cA_\Gamma)_{[\phi]},
\label{eq:spectral-sequence-infinite-order-good}
\end{equation}
where $H_\bt(\sC)$ is the ordinary homology of $\sC$. If  the $b$-differential of $\sC$ is zero, then $E^2_{p,q}= H_p(\OG_\phi, \sC_q)$ and the $E_2$-level differential is given by $(-1)^pB(u_\phi \frown -): H_p(\OG_\phi, \sC_q) \longrightarrow H_{p-2}(\OG_\phi, \sC_{q+1})$. 
\end{corollary}

\begin{remark}\label{rmk:infinite.FT-sequence-good}
 For $\sC=C^\phi(\cA)$ and $\alpha=\op{id}$ the spectral sequence~(\ref{eq:spectral-sequence-infinite-order-good}) specializes to the spectral sequence of Feigin-Tsygan in the infinite order case (Theorem~4.1.1 and Remark~4.1.2 of~\cite{FT:LNM87}).
\end{remark}

The module of group cochains $C^\bt(\OG_\phi,k)$ is a differential graded ring under the usual cup product. The usual cap product then provides us with an associative differential graded action of $C^\bt(\OG_\phi, k)$ on $C_\bt(\OG_\phi)$. By equivariance the cap product extends to a graded bilinear map, 
\begin{equation}
- \frown - : C^\bt(\OG_\phi) \times \Tot_\bt\left(C^\sigma(\OG_\phi, \sC)\right) \longrightarrow \Tot_\bt\left(C^\sigma(\OG_\phi, \sC)\right).
\label{eq:infinite.cap-product-TotCsOGsC}
\end{equation}
We obtain a graded differential action of $C^\bt(\OG_\phi,k)$ on $\Tot_\bt(C^\sigma(\OG_\phi, \sC))$. Combining this with Theorem~\ref{thm:infinite.HCAGphi} we then arrive at the following corollary.

\begin{corollary}\label{cor:infinite.cohom-action-good}
 Under the assumptions of Theorem~\ref{thm:infinite.HCAGphi}, the quasi-isomorphisms~(\ref{eq:infinite.quasi-isom-CAG}) and the cap product~(\ref{eq:infinite.cap-product-TotCsOGsC}) give rise to a graded action of the cohomology ring $H^\bt(\OG_\phi, k)$ on $\HC_\bt(\cA_\Gamma)_{[\phi]}$. The  periodicity operator  $\HC_\bt(\cA_\Gamma)_{[\phi]}$ is given by the action of the Euler class $e_\phi$. In particular, $\HP(\cA_\Gamma)_{[\phi]}=0$ whenever $e_\phi$ is nilpotent in $H^\bt(\OG_\phi,k)$. 
\end{corollary}

\begin{remark}\label{rmk:infinite.Nistor-good}
Nistor~\cite{Ni:InvM90} proved that, for any infinite order element $\phi \in \Gamma$,  the cyclic homology  $\HC_\bt(\cA_\Gamma)_{[\phi]}$ is a module over $H^\bt(\OG_\phi, k)$ and the action of $e_\phi$ gives the periodicity operator. Therefore, Corollary~\ref{cor:infinite.cohom-action-good} provides us with a simple derivation of Nistor's result when the action of $\phi$ on $\cA$ is good. In this case, we actually get a more precise result since the cap product~(\ref{eq:infinite.cap-product-TotCsOGsC}) already provides us with a differential graded action at the level of chains. 
\end{remark}

\subsection{General infinite order actions}
For a general action of an infinite order $\phi \in \Gamma$ we have the following result.

\begin{theorem}\label{thm:infinite.HCAGphi-FT-sequence-general}
 Let $\phi \in \Gamma$ have infinite order, and set $\overline{\Gamma}_\phi = \Gamma_\phi/ \brak\phi$. In addition, let $\alpha: C^\phi_\bt(\cA)\rightarrow \sC_\bt$ be a quasi-isomorphism of $\phi$-parachain complexes. Then we have the following quasi-isomorphisms of cyclic complexes, 
  \begin{equation}
\Tot_\bt\left(C^\phi({\Gamma}_\phi,\sC)\right)^\natural  \xleftarrow{1\otimes \alpha}
 \Tot_\bt\left(C^\phi(\Gamma_\phi,\cA)\right)^\natural\xrightleftharpoons[\AW^\natural]{\shuffle^\natural} \Diag_\bt\left(C^\phi(\Gamma_\phi,\cA)\right)^\natural \xrightarrow{\mu_\phi} C_\bt(\cA_\Gamma)_{[\phi]}^\natural.
 \label{eq:infinite.quasi-isom-CAG-general} 
\end{equation}
These quasi-isomorphisms and the filtration by columns of $\Tot_\bt(C^\phi(\Gamma_\phi)^\natural\otimes_{\Gamma_\phi} \sC)= \Tot_\bt(C^\phi(\Gamma_\phi)\otimes_{\Gamma_\phi} \sC)^\natural$ give rise to a spectral sequence, 
\begin{equation}
E^2_{p,q}=H_p(\overline{\Gamma}_\phi,H_q(\sC)) \Longrightarrow \HC_{p+q}(\cA_\Gamma)_{[\phi]}. 
\label{eq:infinite.FT-spectral-sequence}
\end{equation}
\end{theorem}
 \begin{proof}[Sketch of Proof] 
 The quasi-isomorphisms~(\ref{eq:infinite.quasi-isom-CAG-general}) follow from 
Proposition~\ref{prop:functoriality-CphiAsC} and Proposition~\ref{prop:quasi-isomorphism-CphiGA-CAGphi}. They provide us with an isomorphism $\HC_\bt(\cA_\Gamma)_{[\phi]}\simeq \HC_\bt(\Tot_\bt(C^\phi(\Gamma_\phi,\sC)))$. By Proposition~\ref{prop:triang-S-mod.spectral-sequences-cyclindrical} the cyclic homology $ \HC_\bt(\Tot_\bt(C^\phi(\Gamma_\phi,\sC)))$ is computed by  a spectral sequence with $E^2$-term $\HC_p(H_q^v(C^\phi(\Gamma, \sC)))$. In this case  $H_q^v(C^\phi(\Gamma_\phi, \sC))$ is just the mixed complex $C^\phi(\Gamma_\phi, H_q(\sC))$. By Lemma~\ref{lem:homotopy-inverses-sM} the cyclic homology of $C^\phi(\Gamma_\phi, H_q(\sC))$ is identified with $H_p(\OG_\phi, H_q(\sC))$. This gives the spectral sequence~(\ref{eq:infinite.FT-spectral-sequence})
\end{proof}

\begin{remark}
  In the same way as in Remark~\ref{rmk:infinite.FT-sequence-good}, the spectral sequence~(\ref{eq:infinite.FT-spectral-sequence}) gives back the spectral sequence of Feigin-Tsygan~\cite{FT:LNM87}. Therefore, Theorem~\ref{thm:infinite.HCAGphi-FT-sequence-general} gives a simple derivation of that spectral sequence in the infinite order case when $k\supset \Q$.
\end{remark}

The results of Nistor~\cite{Ni:InvM90} alluded to in Remark~\ref{rmk:infinite.Nistor-good} involve various difficult homological algebra arguments. We shall now explain how to obtain the results of Nistor as a consequence of a general and simple coproduct construction for paracyclic modules. 

Recall that the ordinary Alexander-Whitney map provides us with a coassociative coproduct for simplicial modules. For instance, consider the diagonal map $\delta: C_\bt(\Gamma,k)\rightarrow \Diag_\bt( C(\Gamma,k)\otimes C(\Gamma,k))$ given by $\delta(\eta)=\eta\otimes \eta$, $\eta \in C_m(\Gamma,k)$. This is a map of simplicial modules. Composing it with the Alexander-Whitney map we obtain a coassociative graded coproduct $\Delta:=\AW \circ \delta:C_\bt(\Gamma,k) \rightarrow \Tot_\bt( C(\Gamma,k)\otimes C(\Gamma,k))$. By duality this gives rise to the standard cup product on the group cohomology complex $(C^\bt(\Gamma,k),\partial)$. 

Suppose that $\phi$ is a central element of $\Gamma$.  
Let $\delta: C^\phi(\Gamma) \rightarrow \Diag_\bt( C^\phi(\Gamma,k)\otimes C^\phi(\Gamma))$ be the graded $k\Gamma$-module map given by
\begin{equation*}
\delta (\psi_0, \ldots, \psi_m)= [(\psi_0,\ldots,\psi_m)\otimes_\Gamma 1]\otimes(\psi_0,\ldots,\psi_m), \qquad \psi_j \in \Gamma. 
\end{equation*}
 Here the action of $\Gamma$ on $C_\bt(\Gamma,k)\otimes C_\bt(\Gamma)$ is on the second factor. We obtain a map of parachain complexes, and so this gives rise to an 
$S$-map $C^\phi(\Gamma)^\natural \rightarrow \Diag(C^\phi(\Gamma,k)\otimes C^\phi(\Gamma))$. Composing this map with the bi-paracyclic Alexander-Whitney map from Proposition~\ref{prop:bi-paracyclic-EZ} we then obtain an $S$-map $\AW^\natural \circ \delta: C^\phi(\Gamma)^\natural \rightarrow \Tot(C^\phi(\Gamma,k)\otimes C^\phi(\Gamma))^\natural$.

Let $C^\flat(\OG,k)$ be the mixed complex $(C_\bt(\OG), \partial,0)$. As $\varepsilon \wpi B_\phi=0$ we get a parachain bicomplex map 
$(\varepsilon \wpi)\otimes 1: C^\phi_\bt (\Gamma,k)\otimes C^\phi_\bt(\Gamma)\rightarrow C^\flat_\bt(\OG,k)\otimes C^\phi_\bt(\Gamma)$. This then gives rise to a parachain 
$S$-module map $(\varepsilon \wpi)\otimes 1: \Tot(C^\phi(\Gamma,k)\otimes C^\phi(\Gamma))^\natural \rightarrow \Tot(C^\flat(\OG,k)\otimes C^\phi(\Gamma))^\natural$. As explained in Section~\ref{sec:triangular-S-module}, the para-$S$-module $\Tot(C^\flat(\OG,k)\otimes C^\phi(\Gamma))^\natural$ can be realized as the total para-$S$-module of some vertical triangular para-$S$-module. In the case of $\Tot(C^\flat(\OG,k)\otimes C^\phi(\Gamma))^\natural$ that triangular para-$S$-module is the tensor product of the chain complex $(C(\OG,k),\partial)$ with the para-$S$-module $C^\phi(\Gamma)^\natural$. Denote by $C(\OG,k)\otimes C^\phi(\Gamma)^\natural$ this tensor product, 
we obtain a para-$S$-module map $(\varepsilon \wpi)\otimes 1: \Tot_\bt(C^\phi_\bt (\Gamma,k)\otimes C^\phi(\Gamma))^\natural 
\rightarrow \Tot_\bt(C^\flat(\OG,k)\otimes C^\phi(\Gamma)^\natural)$. Composing it with the $S$-map $\AW^\natural \circ \delta$ above we obtain a para-$S$-module map, 
\begin{equation*}
\Delta^\natural:=((\varepsilon \wpi)\otimes 1)\AW^\natural \circ \delta: C^\phi(\Gamma)^\natural \longrightarrow \Tot_\bt(C^\flat(\OG,k)\otimes C^\phi(\Gamma)^\natural). 
\end{equation*}
It then follows that, for any $\phi$-parachain complex $\sC$, we obtain an $S$-module map,
\begin{equation*}
\Delta^\natural\otimes 1: \Tot_\bt\left(C^\phi(\Gamma,\sC)\right)^\natural \longrightarrow \Tot_\bt\left(C^\flat(\OG,k)\otimes \Tot (C^\phi(\Gamma,\sC))^\natural\right). 
\end{equation*}
Combining this with the duality between $C_\bt(\OG,k)$ and $C^\bt(\OG,k)$ then provides us with a differential graded bilinear map,
\begin{equation}
\triangleright  : C^\bt(\OG,k)\times \Tot_\bt\left(C^\phi(\Gamma,\sC)\right)^\natural \longrightarrow  \Tot_\bt\left(C^\phi(\Gamma,\sC)\right)^\natural. 
\label{eq:infinite.action}
\end{equation}
More precisely, given any cochain $u\in C^p(\OG,k)$, $p\geq 0$, and chains $\eta \in C_\bt^\phi(G)^\natural$ and $\xi \in \sC_\bt$, we have
\begin{equation*}
u \triangleright (\eta \otimes_{\Gamma} \xi)= \left[ (u\otimes 1)\Delta^\natural \eta\right] \otimes_\Gamma \xi \in C^\phi_{\bt -p}(\Gamma)^\natural \otimes_\Gamma \sC_\bt.
\end{equation*}

To understand the relationship between $\triangleright$ and the usual cap product it is convenient to construct a twisted version of the antisymmetrization map $\varepsilon$ as follows. For $p\leq m$ let $\tau_\phi^{(p)}:C_m(\Gamma)\rightarrow C_m(\Gamma)$ and $N_\phi^{(p)}:C_m(\Gamma)\rightarrow C_m(\Gamma)$ be the $k\Gamma$-module maps given by
\begin{gather*}
 \tau_\phi^{(p)}(\psi_0,\ldots, \psi_m)=(-1)^p(\phi^{-1}\psi_p,\psi_0,\ldots, \widehat{\psi_p}, \ldots, \psi_m), \quad \psi_j\in \Gamma, \\
 N_\phi^{(p)}=1+\tau_\phi^{(p)}+ \cdots + \left( \tau_\phi^{(p)}\right)^p.   
\end{gather*}
Note that $\tau_\phi^{(m)}$ agrees with $\tau_\phi=(-1)^mt_\phi$ on $C_m^\phi(\Gamma)$. We then define $\varepsilon_\phi : C_\bt(\Gamma)\rightarrow C_\bt(\Gamma)$ by 
\begin{equation*}
\varepsilon_\phi=\frac{1}{(m+1)!} N^{(0)}_\phi N^{(0)}_\phi \cdots N^{(m)}_\phi \qquad \text{on $C_m(\Gamma)$}. 
\end{equation*}
We observe that $\wpi \varepsilon_\phi= \wpi \varepsilon= \varepsilon \wpi$. 

Let $\sM$ be a $\phi$-invariant $k\Gamma$-module. It can be checked that $\varepsilon_\phi \otimes 1: (C_\bt(\Gamma,\sM),\partial) \rightarrow (C_\bt(\Gamma,\sM),\partial)$ is a chain map. Moreover, the graded $k\Gamma$-module map $(\iota \varepsilon_\phi)\otimes 1:  C_\bt(\Gamma,\sM)\rightarrow C_\bt^\phi(\Gamma,\sM)^\natural$ descends to a chain map $((\iota \varepsilon_\phi)\otimes 1)^\lambda: C_\bt(\Gamma,\sM)^\lambda\rightarrow C_\bt^\phi(\Gamma,\sM)^\natural$ such that 
$(\wepsn\otimes 1) ((\iota \varepsilon_\phi)\otimes 1)^\lambda = (\wepsn\otimes 1) (\iota \otimes 1)^\lambda$. As $  (\iota \otimes 1)^\lambda (\gamma \otimes 1)$ is a homotopy inverse of $(\wepsn\otimes 1)$, it then follows that $((\iota \varepsilon_\phi)\otimes 1)^\lambda( \gamma \otimes 1)$ is a homotopy inverse as well.

\begin{lemma}\label{eq:infinite.action-cap-product}
 For any cocycle $u\in C^\bt(\OG,k)$, we have 
 \begin{equation*}
(\wepsn\otimes 1)(u \triangleright - )\left((\iota \varepsilon_\phi)\otimes 1\right)^\lambda( \gamma \otimes 1)= u \frown - \qquad \text{on $H_\bt(\OG, \sM)$}. 
\end{equation*}
\end{lemma}
 \begin{proof}[Sketch of Proof] The bulk of the proof is to prove the equality,
\begin{equation}
(1\otimes \wepsn \otimes 1) (\Delta^\natural \otimes 1)\left[(\iota \varepsilon_\phi)\otimes 1\right]^\lambda( \gamma \otimes 1)=
 (\varepsilon \otimes \varepsilon \otimes 1) \left[ (\AW \delta)\otimes 1\right]  .
 \label{eq:infinite.coproduct-AW}
\end{equation}
Once this equality is proved the lemma follows by using the facts that $u \frown - =(u\otimes 1) \AW \delta$ and  the chain maps $\varepsilon$ and $ (\varepsilon \wpi^\lambda \gamma)$ both are $\partial$-homotopic to identity maps, and hence induce identity maps on homology. 

To prove~(\ref{eq:infinite.coproduct-AW}) we observe that its left-hand side is equal to
\begin{multline*}
 (1\otimes \wepsn \otimes 1)  \left[(\varepsilon \wpi)\otimes 1 \otimes 1\right] (\AW^\natural \otimes 1) (\delta \otimes 1)
 \left[(\iota \varepsilon_\phi)\otimes 1\right]^\lambda( \gamma \otimes 1)\\
=   (\varepsilon \otimes \varepsilon \otimes 1)  (\pi^\natural \otimes \pi^\natural \otimes 1)  (\AW^\natural \otimes 1) (\delta \otimes 1) (\wpi\otimes 1)\left[(\iota \varepsilon_\phi)\otimes 1\right]^\lambda( \gamma \otimes 1)\\
=  (\varepsilon \otimes \varepsilon \otimes 1)  (\pi^\natural \otimes \pi^\natural \otimes 1)  (\AW^\natural \otimes 1) (\delta \otimes 1) 
\left[ (\varepsilon \wpi^\lambda \gamma)\otimes 1\right]. 
\end{multline*}
We then identify the lowest side with the right-hand side of~(\ref{eq:infinite.coproduct-AW}) by using the following two facts. First, the  antisymmetrization map on 
$C_\bt(\OG)$ is annihilated by all degenerate chains, including the chain in the range of the extra degeneracy. Second, the bi-paracyclic 
Alexander-Whitney map $\AW^\natural$ is precisely built so that its zeroth degree component is given by the ordinary Alexander-Whitney map. This completes the proof. 
\end{proof}

Lemma~\ref{eq:infinite.action-cap-product} is the key ingredient in the proof of the following result.

\begin{proposition}\label{prop:infinite.action-sC}
 Let $\sC$ be a $\phi$-parachain complex. Then the bilinear map~(\ref{eq:infinite.action}) descends to an associative graded action, 
 \begin{equation*}
 \triangleright  : H^\bt(\OG, k) \times \HC_\bt \left(\Tot(C^\phi(\Gamma,\sC))\right) \longrightarrow  \HC_\bt \left(\Tot(C^\phi(\Gamma,\sC))\right). 
\end{equation*}
Moreover, the periodicity operator of $ \HC_\bt \left(\Tot(C^\phi(\Gamma,\sC))\right)$ is given by the action of the Euler class of the extension 
$1\rightarrow \brak\phi \rightarrow \overline{\Gamma} \rightarrow \Gamma \rightarrow 1$. 
\end{proposition}
\begin{proof}[Sketch of Proof] It is enough to check the result at the $E^2$-level of the spectral sequence of Proposition~\ref{prop:triang-S-mod.horizontal-SC} for $C^\phi(\Gamma,\sC)$. Its $E^2$-term is 
$E^2_{p,q}=\HC_p (C^\phi(\Gamma,H_q(\sC)))$. As before $H_q(\sC)$ is a $\phi$-invariant $k\Gamma$-module, and so by Lemma~\ref{lem:homotopy-inverses-sM}  the quasi-isomorphism $\wepsn \otimes 1$ identifies $E^2_{p,q}$ with $H_p(\OG, H_q(\sC))$. Combining this with Lemma~\ref{eq:infinite.action-cap-product} we then see that at the $E^2$-level the bilinear map $ \triangleright $ is given by the ordinary cap product on $H_\bt(\OG, H_q(\sC))$. That cap product is associative and by Lemma~\ref{lem:infinite-order.homotopy-S} the cap product with the Euler class $e_\phi$ agrees with the periodicity operator. It then follows that $ \triangleright $ descends to a graded action of $H^\bt(\OG, k)$ on $\HC_\bt(\Tot(C^\phi(\Gamma,\sC)))$ and the action of $e_\phi$ gives the periodicity operator.
\end{proof}

Specializing Proposition~\ref{prop:infinite.action-sC} to $\sC=C^\phi(\cA)$ and using Proposition~\ref{prop:quasi-isomorphism-CphiGA-CAGphi} gives the following result.

\begin{theorem}[compare~\cite{Ni:InvM90}]\label{thm:infinite.action-HC}
 Let $\phi \in \Gamma$ have infinite order, and set $\overline{\Gamma}_\phi = \Gamma_\phi/ \brak\phi$.
 \begin{enumerate}
\item[(1)] The quasi-isomorphisms~(\ref{eq:quasi-isomorphism-CphiGA-CAGphi}) and the bilinear map~(\ref{eq:infinite.action}) for $\sC=C^\phi(\cA)$ give rise to an action of the cohomology ring $H^\bt(\OG_\phi,k)$ on the cyclic homology $\HC_\bt(\cA_\Gamma)_{[\phi]}$.

\item[(2)] The periodicity operator of $\HC_\bt(\cA_\Gamma)_{[\phi]}$  is given by  the action of the Euler class $e_\phi \in H^2(\overline{\Gamma}_\phi,k)$ of the extension $1\rightarrow \brak\phi \rightarrow \overline{\Gamma}_\phi \rightarrow \Gamma_\phi \rightarrow 1$. In particular, $\HP_\bt(\cA_\Gamma)_{[\phi]}=0$ whenever $e_\phi$ is nilpotent in $H^\bt(\Gamma_\phi,k)$. 
\end{enumerate}
\end{theorem}

\begin{remark}
The nilpotence of $e_\phi$ is closely related to the  Bass and idempotent conjectures (see, e.g., \cite{Em:InvM98, Ji:KT95}). In particular, $e_\phi$ is rationally nilpotent for every infinite order element $\phi \in \Gamma$ whenever $\Gamma$ belongs to one of the following classes of groups: free products of Abelian groups, hyperbolic groups, and arithmetic groups. 
\end{remark}

%\part{Locally Convex Algebras.  Group Actions on Manifolds and Varieties}
\section{Actions on Locally Convex Algebras}\label{sec:LCA}
In this section, we explain how to extend the results of the previous sections to the setting of group actions on locally convex algebras. Throughout this section we let 
$\Gamma$ be a group which acts by continuous algebra automorphisms on a unital locally convex algebra $\cA$. 

\subsection{Splitting along conjugacy classes} 
Let $\cA_\Gamma$ be the crossed-product algebra $\cA\rtimes \Gamma$. This is the unital algebra with generators $a\in \cA$ and $u_\phi$, $\phi \in \Gamma$ subject to the relations,
\[
 a^0u_{\phi_0} a^1u_{\phi_1} =a^0(\phi_0^{-1} a^1)u_{\phi_0\phi_1}, \qquad a^j \in \cA, \ \phi_j\in \Gamma. 
 \] 
As a vector space $\cA_\Gamma$ is  $\C \Gamma \otimes \cA$. We endow  $\cA_\Gamma$ with the weakest locally convex topology with respect to which the linear embeddings $\cA\ni a \rightarrow au_\phi\in \cA_\Gamma$, $\phi \in \Gamma$, are continuous. With respect to this topology $\cA_\Gamma$ is a locally convex algebra. 

Given any $\phi \in \Gamma$, we denote by $[\phi]$ its conjugacy class in $\Gamma$. In addition,  we let $\bC_m(\cA_\Gamma)_{[\phi]}$, $m\geq 0$, be the closure of the subspace of $\bC_m(\cA_\Gamma)$ spanned by 
 tensor products $a^0u_{\phi_0}\otimes \cdots \otimes a^mu_{\phi_m}$, with $a^0,\ldots, a^m$ in $\cA$ and $\phi_0,\ldots, \phi_m$ in $\Gamma$ such that $\phi_0 \cdots \phi_m\in [\phi]$. This gives rise to a cyclic subspace $\bC(\cA_\Gamma)_{[\phi]}$. We then have the following direct sum decomposition of cyclic spaces,
 \begin{equation}
 \bC(\cA_\Gamma)= \bigoplus  \bC(\cA_\Gamma)_{[\phi]},
 \label{eq:direct-sum}
\end{equation}
 where the sum goes over all conjugacy classes in $\Gamma$.  As in Section~\ref{sec:conj-classes}, this provides us with a 
 corresponding decomposition of cyclic complexes along with a canonical inclusion of periodic cyclic complexes. Let us denote by $\bHC_\bt(\cA_\Gamma)_{[\phi]}$ (resp., $\bHP_\bt(\cA_\Gamma)_{[\phi]}$) the cyclic homology (resp., periodic cyclic homology)
of $\bC(\cA_\Gamma)_{[\phi]}$. Then we have the following splitting and canonical inclusion of graded $k$-modules, 
\begin{equation*}
 \bHC(\cA_\Gamma)= \bigoplus \bHC_\bt(\cA_\Gamma)_{[\phi]} \qquad  \text{and} \qquad  \bigoplus \bHP_\bt(\cA_\Gamma)_{[\phi]}\subset  \bHP(\cA_\Gamma).
 \label{eq:LCA.splitting-HCHP}
 \end{equation*}
Moreover, the inclusion above is actually onto when $\Gamma$ has a finite number of conjugation classes.

Let $\phi\in \Gamma$ and denote by $\Gamma_\phi$ its centralizer in $\Gamma$.  The structural operators $(d_\phi, s,t_\phi)$ of the paracyclic $\C\Gamma_\phi$-module $C^\phi(\cA)$ uniquely extend to continuous operators on $\bC_\bt(\cA)$ so that we obtain a paracyclic $\C\Gamma_\phi$-module $\bC^\phi(\cA):=(\bC_\bt(\cA), d_\phi, s, t_\phi)$. 
We denote by $\bC^\phi(\Gamma_\phi, \cA)$ the cylindrical space $C^\phi(\Gamma_\phi, \bC^\phi(\cA))$ as defined in Section~\ref{sec:CphiGA}. This is just the tensor product over $\Gamma_\phi$ of the paracyclic $\C\Gamma_\phi$-modules $C^\phi(\Gamma_\phi)$ and $\bC^\phi(\cA)$. The space of $(p,q)$-chains is $\bC^\phi_{p,q}(\Gamma_\phi, \cA):=C_p(\Gamma_\phi)\otimes_{\Gamma_\phi} \bC_q(\cA)$. We equip it with the weakest locally convex topology with respect to which the linear embeddings $\bC_q(\cA)\ni \xi \rightarrow (\psi_0,\ldots, \psi_p)\otimes_{\Gamma_\phi} \xi\in \bC^\phi_{p,q}(\Gamma_\phi, \cA)$, $\psi_j\in \Gamma_\phi$, are continuous.
 In Section~\ref{sec:conj-classes} we exhibited a cyclic space embedding and quasi-isomorphism $\mu_\phi: \Diag_\bt(C^\phi(\Gamma_\phi,\cA))\rightarrow C_\bt(\cA_\Gamma)_{[\phi]}$ given by~(\ref{eq:splitting.mu-phi}). 
This embedding is continuous, and so it uniquely extends to a continuous embedding of cyclic spaces $\mu_\phi: \Diag_\bt(\bC^\phi(\Gamma_\phi,\cA))\rightarrow \bC_\bt(\cA_\Gamma)_{[\phi]}$. When $\phi=1$ we obtain an isomorphism of cyclic spaces. Like its restriction to $\Diag_\bt(\bC^\phi(\Gamma_\phi,\cA))$ this is a quasi-isomorphism of cyclic spaces. Therefore, we arrive at the following result.

\begin{proposition}\label{prop:quasi-isomorphism-bCphiGA-bCAGphi}
 Let $\phi \in \Gamma$. Then we have the following quasi-isomorphisms of cyclic complexes,
\begin{equation}
\Tot_\bt\left(\bC^\phi(\Gamma_\phi,\cA)\right)^\natural\xrightleftharpoons[\AW^\natural]{\shuffle^\natural} \Diag_\bt\left(\bC^\phi(\Gamma_\phi,\cA)\right)^\natural \xrightarrow{\mu_\phi} \bC_\bt(\cA_\Gamma)_{[\phi]}^\natural. 
\label{eq:LCA.quasi-isomorphism-CphiGA-CAGphi}
\end{equation}
There are similar results  at the level of the corresponding periodic cyclic complexes. 
\end{proposition}

The mixed complex $\Tot_\bt\left(\bC^\phi(\Gamma_\phi,\cA)\right)^\natural$ can be studied exactly in the same way as in Sections~\ref{sec:finite}--\ref{sec:infinite-order}. Therefore, the main results of these sections for $C(\cA_\Gamma)_{[\phi]}$ apply \emph{mutatis mutandis} to its completion $\bC(\cA)_{[\phi]}$ by replacing the twisted cyclic space $C^\phi(\cA)$ by its completion $\bC^\phi(\cA)$. 

When $\Gamma_\phi$ is finite, by combining Proposition~\ref{prop:quasi-isomorphism-bCphiGA-bCAGphi} with Proposition~\ref{prop:functoriality-CphiAsC} and  Proposition~\ref{prop:quasi-isom-Gfinite} we arrive at the following result.

\begin{theorem}\label{thm:LCA.bHCAGphi-finite} 
Suppose that $\Gamma_\phi$ is finite and we are given a quasi-isomorphism of $\phi$-parachain complexes $\alpha: \bC^\phi_\bt(\cA) \rightarrow \sC_\bt$. 
\begin{enumerate}
\item[(1)] We have explicit quasi-isomorphisms of cyclic complexes, 
\[
  \sC^{\Gamma_\phi,\natural}_\bt  \xleftarrow{\pi_0\otimes \alpha}
 \Tot_\bt\left(\bC^\phi(\Gamma_\phi,\cA)\right)^\natural\xrightleftharpoons[\AW^\natural]{\shuffle^\natural} \Diag_\bt\left(\bC^\phi(\Gamma_\phi,\cA)\right)^\natural \xrightarrow{\mu_\phi} \bC_\bt(\cA_\Gamma)_{[\phi]}^\natural.
 \] 
Moreover, if  $\alpha$ has a quasi-inverse (resp., homotopy inverse) $\beta:  \sC_\bt \rightarrow \bC^\phi_\bt(\cA)$, then 
$  (1\otimes \beta) \iota_{\Gamma_\phi}^\natural$ is a quasi-inverse (resp., homotopy inverse) of $\pi_0\otimes \alpha$.

\item[(2)]  There are analogous statements for the corresponding periodic cyclic complexes. 

\item[(3)] This gives rise to isomorphisms, 
\begin{equation*}
\bHC(\cA_\Gamma)_{[\phi]}\simeq \HC_\bt(\sC^{\Gamma_\phi}) \quad \text{and} \quad  \bHP(\cA_\Gamma)_{[\phi]}\simeq 
\HP_\bt(\sC^{\Gamma_\phi}). 
\end{equation*}
\end{enumerate}
\end{theorem}

When $\phi$ has finite order we can combine Proposition~\ref{prop:quasi-isomorphism-bCphiGA-bCAGphi} with Proposition~\ref{prop:functoriality-CphiAsC} and  Proposition~\ref{prop:quasi-isom-Gfinite} to get the following result.

 \begin{theorem}\label{thm:LCA.bHCAGphi-finite-order}
Let $\phi \in \Gamma$ have finite order, and suppose we are given a quasi-isomorphism of $\phi$-parachain complexes 
 $\alpha: \bC^\phi_\bt(\cA) \rightarrow \sC_\bt$. 
 \begin{enumerate}
\item[(1)] We have explicit quasi-isomorphisms of cyclic complexes, 
\begin{equation}
  \Tot_\bt\left(C^\flat(\Gamma_\phi, \sC^\phi)\right)^\natural  \xleftarrow{(\varepsilon \nu_\phi)\otimes \alpha}
 \Tot_\bt\left(\bC^\phi(\Gamma_\phi,\cA)\right)^\natural\xrightleftharpoons[\AW^\natural]{\shuffle^\natural} \Diag_\bt\left(\bC^\phi(\Gamma_\phi,\cA)\right)^\natural \xrightarrow{\mu_\phi} \bC_\bt(\cA_\Gamma)_{[\phi]}^\natural.
 \label{eq:LCA.finite-order.quasi-isom-CAGphi}
 \end{equation}
Moreover, if $\alpha$ has a quasi-inverse (resp., homotopy inverse) $\beta:  \sC_\bt \rightarrow \bC^\phi_\bt(\cA)$, then 
$ (\varepsilon \nu_\phi)^\flat \otimes \beta$ is a quasi-inverse (resp., homotopy inverse) of $(\varepsilon \nu_\phi)\otimes \alpha$. 

\item[(2)] There are analogous statements for the corresponding periodic cyclic complexes. 

\item[(3)] This provides us with isomorphisms,
\begin{equation*}
\bHC(\cA_\Gamma)_{[\phi]}\simeq \HC_\bt( \Tot(C^\flat(\Gamma_\phi,\sC^\phi))) \quad \text{and} \quad 
\bHP(\cA_\Gamma)_{[\phi]}\simeq \HP_\bt( \Tot(C^\flat(\Gamma_\phi,\sC^\phi))). 
\end{equation*}
\end{enumerate}
\end{theorem}

In the same way as with Corollary~\ref{cor:finite-order.HCAGphi-spectral-sequences}, from Theorem~\ref{thm:LCA.bHCAGphi-finite-order} we deduce the following statement.

\begin{corollary}
 Suppose that the assumptions of Theorem~\ref{thm:LCA.bHCAGphi-finite-order} hold. 
 \begin{enumerate}
\item[(1)] The filtrations by rows and columns of $\Tot_\bt(C^\flat(\Gamma_\phi)^\natural\otimes_{\Gamma_\phi} \sC^\phi)=  \Tot_\bt\left(C^\flat(\Gamma_\phi, \sC^\phi)\right)^\natural$ and the quasi-isomorphisms~(\ref{eq:LCA.finite-order.quasi-isom-CAGphi}) give rise to spectral sequences, 
\begin{gather*}
  {~}^{I}\! E^{2}_{p,q}=\HC_p\left(H_q(\Gamma_\phi, \sC^\phi)\right) \Longrightarrow 
        \bHC_{p+q}( \cA_{\Gamma})_{[\phi]},  
%        \label{eq:spectral-sequence-finite-order1}
        \\
  {~}^{I\!I}\! E^{2}_{p,q}=H_p(\Gamma_\phi, \HC_q(\sC^\phi)) \Longrightarrow 
        \bHC_{p+q}( \cA_{\Gamma})_{[\phi]}. 
%       \label{eq:spectral-sequence-finite-order2}       
\end{gather*}
Here $H_q(\Gamma_\phi, \sC^\phi)$ is the mixed complex $(H_q(\Gamma_\phi, \sC^\phi_\bt), b,B)$.

\item[(2)] The filtration by columns of $\Tot_\bt(C^\flat(\Gamma_\phi)^\natural\otimes_{\Gamma_\phi} \sC^\phi)=  \Tot_\bt\left(C^\flat(\Gamma_\phi, \sC^\phi)\right)^\natural$ and the quasi-isomorphisms~(\ref{eq:LCA.finite-order.quasi-isom-CAGphi}) give rise to a spectral sequence, 
        \begin{equation*}
{~}^{I\!I\!I}\! E^{2}_{p,q} \Longrightarrow 
        \bHC_{p+q}( \cA_{\Gamma})_{[\phi]}, \quad \text{where}\ {~}^{I}\! E^{2}_{p,q}=H_p(\Gamma_\phi, H_q(\sC^\phi))\oplus H_p(\Gamma_\phi, H_q(\sC^\phi)) \oplus \cdots,  
%                \label{eq:spectral-sequence-finite-order3}
\end{equation*}
and $H_\bt(\sC^\phi)$ is the ordinary homology of $\sC^\phi$. 
\end{enumerate}
\end{corollary}

Finally, when $\phi$ has infinite order, we shall say that the action of $\phi$ on $\cA$ is \emph{topologically good} when $\bC^\phi(\cA)$ is a good parachain complex. Bearing this in mind, by combining Proposition~\ref{prop:quasi-isomorphism-bCphiGA-bCAGphi} with Proposition~\ref{prop:functoriality-CphiAsC} and  Proposition~\ref{prop:infinite-order.quasi-isom-sC} we arrive at the following result.

\begin{theorem}\label{thm:LCA.infinite-good-action}
 Let $\phi \in \Gamma$ have infinite order, and set $\overline{\Gamma}_\phi = \Gamma_\phi/ \brak\phi$. Assume that the action of $\phi$ on $\cA$ is topologically good, and let 
 $\alpha: C^\phi_\bt(\cA)\rightarrow \sC_\bt$ be a parachain quasi-isomorphism, where $\sC$ is a $\phi$-invariant mixed complex. Then we have the following quasi-isomorphisms, 
 \begin{equation}
 \Tot_\bt (C^\sigma(\overline{\Gamma}_\phi ,\sC))  \xleftarrow{\theta(1\otimes \alpha)}
 \Tot_\bt\left(\bC^\phi(\Gamma_\phi,\cA)\right)^\natural\xrightleftharpoons[\AW^\natural]{\shuffle^\natural} \Diag_\bt\left(C^\phi(\Gamma_\phi,\cA)\right)^\natural \xrightarrow{\mu_\phi} \bC_\bt(\cA_\Gamma)_{[\phi]}^\natural.
 \label{eq:LCA.infinite.quasi-isom-good-action}
\end{equation}
 This gives rise to an isomorphism, 
\begin{equation*}
\bHC_\bt(\cA)_{[\phi]}\simeq H_\bt(C^\sigma(\overline{\Gamma}_\phi ,\sC)). 
\end{equation*}
Under this isomorphism, the periodicity operator of $\bHC_\bt(\cA)_{[\phi]}$ is given by the cap product,
\begin{equation*}
e_\phi \frown - : H_\bt(\Tot(C^\sigma(\overline{\Gamma}_\phi ,\sC))) \longrightarrow H_{\bt-2}(\Tot(C^\sigma(\overline{\Gamma}_\phi ,\sC))),
\end{equation*}
where $e_\phi\in H^2(\OG_\phi,k)$ is the Euler class of the extension $1\rightarrow \brak\phi \rightarrow \Gamma_\phi \rightarrow    \overline{\Gamma}_\phi\rightarrow 1$. 
\end{theorem}

As with Corollary~\ref{cor:infinite.FT-spectral-sequence-good}, we obtain following consequence of Theorem~\ref{thm:LCA.infinite-good-action}.

\begin{corollary}\label{cor:LCA.infinite-spectral-sequence-good}
 Under the assumptions of Theorem~\ref{thm:LCA.infinite-good-action}, the quasi-isomorphisms~(\ref{eq:LCA.infinite.quasi-isom-good-action}) and the filtration by columns of $\Tot(C^\sigma(\overline{\Gamma}_\phi ,\sC))$ give rise to a spectral sequence, 
  \begin{equation*}
E^2_{p,q}= H_p(\OG_\phi, H_q(\sC)) \Longrightarrow \bHC_{p+q}(\cA_\Gamma)_{[\phi]},
%\label{eq:LCA.infinite.FT-spectral-sequence-good}
\end{equation*}
where $H_\bt(\sC)$ is the ordinary homology of $\sC$. If  the $b$-differential of $\sC$ is zero, then we have $E^2_{p,q}= H_p(\OG_\phi, \sC_q)$ and the 
$E^2$-differential is $(-1)^pB(u_\phi \frown -): H_p(\OG_\phi, \sC_q) \longrightarrow H_{p-2}(\OG_\phi, \sC_{q+1})$. 
\end{corollary}

In the case of general infinite order actions, in the same way as with Theorem~\ref{thm:infinite.HCAGphi-FT-sequence-general} we obtain the following result.

\begin{theorem}
 Let $\phi \in \Gamma$ have infinite order, and suppose we are given a quasi-isomorphism of parachain complexes 
 $\alpha: C^\phi_\bt(\cA)\rightarrow \sC_\bt$, where $\sC$ is a $\phi$-parachain complex. Then we have the following quasi-isomorphisms of cyclic complexes, 
  \begin{equation*}
\Tot_\bt\left(C^\phi(\Gamma_\phi,\sC)\right)^\natural  \xleftarrow{1\otimes \alpha}
 \Tot_\bt\left(\bC^\phi(\Gamma_\phi,\cA)\right)^\natural\xrightleftharpoons[\AW^\natural]{\shuffle^\natural} \Diag_\bt\left(\bC^\phi(\Gamma_\phi,\cA)\right)^\natural \xrightarrow{\mu_\phi} \bC_\bt(\cA_\Gamma)_{[\phi]}^\natural.
% \label{eq:LCA.infinite.quasi-isom-CAG-general} 
\end{equation*}
These quasi-isomorphisms and the filtration by columns of $\Tot_\bt(C^\phi(\Gamma_\phi)^\natural\otimes_{\Gamma_\phi} \sC)= \Tot_\bt(C^\phi(\Gamma_\phi)\otimes_{\Gamma_\phi} \sC)^\natural$ give rise to a spectral sequence, 
\begin{equation*}
E^2_{p,q}=H_p(\overline{\Gamma}_\phi,H_q(\sC)) \Longrightarrow \bHC_{p+q}(\cA_\Gamma)_{[\phi]}. 
%\label{eq:infinite.FT-spectral-sequence}
\end{equation*}
\end{theorem}

Finally, specializing Proposition~\ref{prop:infinite.action-sC} to $\sC=\bC^\phi(\cA)$ and using Proposition~\ref{prop:quasi-isomorphism-bCphiGA-bCAGphi} we obtain the following statement.

\begin{theorem}\label{thm:LCA.infinite.action-HC}
 Let $\phi \in \Gamma$ have infinite order, and set $\overline{\Gamma}_\phi = \Gamma_\phi/ \brak\phi$.
 \begin{enumerate}
\item[(1)] The quasi-isomorphisms~(\ref{eq:LCA.quasi-isomorphism-CphiGA-CAGphi}) and the bilinear map~(\ref{eq:infinite.action}) for $\sC=\bC^\phi(\cA)$ give rise to an action of the cohomology ring $H^\bt(\OG_\phi,\C)$ on the cyclic homology $\bHC_\bt(\cA_\Gamma)_{[\phi]}$.

\item[(2)] The periodicity operator of $\bHC_\bt(\cA_\Gamma)_{[\phi]}$  is given by  the action of the Euler class $e_\phi$. In particular, $\bHP_\bt(\cA_\Gamma)_{[\phi]}=0$ whenever $e_\phi$ is nilpotent in $H^\bt(\Gamma_\phi,\C)$. 
\end{enumerate}
\end{theorem}

\section{Equivariant Cohomology and Mixed Equivariant Homology}\label{sec:equivariant} 
Assume that $\Gamma$ acts by (smooth) diffeomorphisms on a manifold $M$.  
Let $\Omega(M)=(\Omega^\bt(M), d)$ be the de Rham complex of differential forms on $M$. Recall that the \emph{equivariant cohomology} $H^\bt_\Gamma(M)$ of the $\Gamma$-manifold $M$ is the cohomology of the total complex of Bott's cochain bicomplex,
\begin{equation*}
C_\Gamma(M)=\left(C^{\bt,\bt}_\Gamma(M), \partial,d\right), \qquad \text{where}\ C^{p,q}_\Gamma(M) =C^p(\Gamma, \Omega^q(M)). 
\end{equation*}
As in Section~\ref{sec:CphiGA},  $C^p(\Gamma, \Omega^q(M))$ consists of $\Gamma$-equivariant maps $\omega: \Gamma^{p+1}\rightarrow \Omega^q(M)$.  In other words, $H^\bt_\Gamma(M)$ is the cohomology of the cochain complex 
$(\Tot_\bt (C_\Gamma(M)), d^\dagger)$, where $ \Tot_m(C_\Gamma(M))= \oplus_{p+q=m} C^{p,q}_\Gamma(M)$ and $d^\dagger=\partial+(-1)^p d$ on 
$C^{p,q}_\Gamma(M)$. It is isomorphic to the cohomology of the homotopy quotient $E\Gamma \times_\Gamma M$. 

The \emph{even/odd equivariant cohomology} $H^{\ev/ \odd}_\Gamma(M)$ 
is the cohomology of the complex,
\begin{equation*}
C^{\ev/\odd}_\Gamma(M)=\left(C^{\ev/ \odd}_\Gamma(M),d^\dagger\right), \qquad \text{where}\ 
C^{\ev/ \odd}_\Gamma(M)=  \prod_{p+q\ \textup{even/odd}} C^{p,q}_\Gamma(M).
\end{equation*}
This is a natural receptacle for the construction of equivariant characteristic classes (\emph{cf}.~\cite{Bo:LNM78}). In particular, given any $\Gamma$-equivariant vector bundle $E$ over $M$ we have a well defined equivariant Chern character $\Ch_\Gamma(E)\in H^\ev_\Gamma(M)$ (see~\cite{Bo:LNM78}). 

We also can define the equivariant homology $H^\Gamma_\bt(M)$ of the $\Gamma$-manifold $M$ by using a dual version of Bott's bicomplex. For our purpose we actually need to construct a "mixed" version of equivariant homology. More precisely, we introduce the equivariant mixed bicomplex, 
\begin{equation*}
C(\Gamma, M):= \left(C_{\bt,\bt}(\Gamma, M), \partial,0,0,d\right), \qquad  \text{where}\ C_{p,q}(\Gamma, M)= C_p(\Gamma)\otimes_\Gamma \Omega^q(M). 
\end{equation*}
In other words, $C(\Gamma, M)$ is the mixed bicomplex $C^\flat(\Gamma,\sC)$ of Section~\ref{sec:LCA}  associated with 
the mixed complex $\sC=(\Omega^\bt(M),0,d)$. This gives rise to the mixed complex 
$\Tot(C(\Gamma,M))=(\Tot_\bt(C(\Gamma,M)), \partial, (-1)^p d)$.

\begin{definition}
 The cyclic homology of the mixed complex $\Tot(C(\Gamma,M))$ is called the \emph{mixed equivariant homology} of the $\Gamma$-manifold $M$ and is denoted by $H^\Gamma_\bt(M)^\natural$. Its periodic cyclic homology is called the \emph{even/odd mixed equivariant homology} of $M$ and is denoted by $H^\Gamma_{\ev/\odd}(M)^\sharp$.
\end{definition}

The mixed equivariant homology is the natural receptacle of the cap product between equivariant cohomology and group homology. Namely, the usual cap product $\frown: C^{p,q}_\Gamma(M) \times C_m(\Gamma,\C) \rightarrow C_{m-p,q}(\Gamma, M)$ is compatible with the differentials $\partial$ and $d$, and so it gives rise to a cap product,
\begin{equation}
 \frown: H^{\ev/\odd}_\Gamma(M) \times H_{\ev/\odd}(\Gamma,\C) \longrightarrow H^\Gamma_{\ev/\odd}(M)^\sharp. 
\label{eq:cap-products-evenodd}
\end{equation}
In particular, caping equivariant characteristic classes with group homology provides us with a geometric construction of mixed equivariant homology classes.

\section{Group Actions on Manifolds}\label{sec:manifolds}
In this section, we assume that $\Gamma$ is a group acting by diffeomorphisms on a compact manifold $M$. We get an action of $\Gamma$ by continuous algebra automorphisms on the Fr\'echet algebra $\cA:=C^\infty(M)$. We shall now explain how to use the results of the previous sections for constructing explicit quasi-isomorphisms for the cyclic and periodic homologies of the crossed-product algebra $\cA_\Gamma=\cA\rtimes \Gamma$. 

In what follows, given any $\phi\in \Gamma$, we denote by $M^\phi$ its fixed-point set in $M$. 

\begin{definition}[{see~\cite{Br:AIF87,BN:KT94}}]
 Let $\phi \in \Gamma$. We shall say that the action of $\phi$ on $M$ is \emph{clean} when, for every $x_0\in M^\phi$,  the following conditions are satisfied:
\begin{enumerate}
\item The fixed-point set $M^\phi$ is a submanifold of $M$ near $x_0$. 

\item We have $T_{x_0}M^\phi = \ker (\phi'(x_0)-1)$ and $T_{x_0}M = T_{x_0}M^\phi \oplus \op{ran} (\phi'(x_0)-1)$.  
\end{enumerate}
\end{definition}

\begin{remark}
 The action of $\phi$ on $M$ is clean whenever it preserves an affine connection or even a Riemannian metric. In particular, we always have a clean action when $\phi$ has finite order.
\end{remark}

Let $\phi \in \Gamma$ act cleanly on $M$. For $a=0,1,\ldots, \dim M$, set 
\begin{equation*}
M^\phi_a:=\{x\in M^\phi; \ \dim \ker (\phi'(x)-1)=a\}. 
\end{equation*}
Then $M^\phi_a$ is a submanifold of $M$ of dimension~$a$. We thus obtain a stratification, 
\begin{equation*}
M^\phi=\bigsqcup M^\phi_a. 
\end{equation*}
This enables us to define the de Rham complex $\Omega(M^\phi)= (\Omega^\bt(M^\phi),d)$ as the direct sum of the de Rham complexes $\Omega(M^\phi_a)$. Note also that each component $M_a^\phi$ is preserved by the action of centralizer $\Gamma_\phi$ of $\phi$. We also define the equivariant bicomplex $C_{\Gamma_\phi}(M^\phi)$ and the equivariant mixed bicomplex $C(\Gamma_\phi,M^\phi)$ as the direct sums of the bicomplexes $C_{\Gamma_\phi}(M^\phi_a)$ and $C(\Gamma_\phi,M^\phi_a)$, respectively. This enables us to define the equivariant cohomology $H^\bt_{\Gamma_\phi}(M^\phi)$ and mixed equivariant homology 
$H_\bt^{\Gamma_\phi}(M^\phi)^\natural$ of $M^\phi$. 
We then have a Hochschild-Kostant-Rosenberg map $\alpha^{\phi}:\bC^\phi(\cA) \rightarrow \Omega^{\bt}(M^\phi)$ given by 
\begin{equation}
\alpha^\phi(f^0\otimes \cdots \otimes f^m)=\frac{1}{m!} \sum_a (f^0df^1\wedge \cdots \wedge df^m)_{|M^\phi_a}, \qquad f^j \in \cA.
\label{eq:HKR-map}
\end{equation}
This is a map of parachain complexes.

\begin{proposition}[{Brylinski~\cite{Br:AIF87}, Brylinski-Nistor~\cite{BN:KT94}}]\label{prop:twisted-CHKR}
 Let $\phi \in \Gamma$ acts cleanly on $M$. Then the parachain complex map~(\ref{eq:HKR-map}) above is a quasi-isomorphism. 
\end{proposition}

\begin{remark}
  For $\phi=1$ this result is due to Connes~\cite{Co:NCDG}. 
\end{remark}
 
Combining Proposition~\ref{prop:twisted-CHKR} with the results of Sections~\ref{sec:LCA} enables us to construct quasi-isomorphisms for the cyclic space $\bC(\cA_\Gamma)_{[\phi]}$ as follows. 

First, when the centralizer $\Gamma_\phi$ is finite, we have the following result.

\begin{theorem} Suppose that $\Gamma_\phi$ is finite. Then the following are quasi-isomorphisms, 
\[
  \Omega(M^\phi)^{\Gamma_\phi,\natural}_\bt  \xleftarrow{\pi_0\otimes \alpha^\phi}
 \Tot_\bt\left(\bC^\phi(\Gamma_\phi,\cA)\right)^\natural\xrightleftharpoons[\AW^\natural]{\shuffle^\natural} \Diag_\bt\left(\bC^\phi(\Gamma_\phi,\cA)\right)^\natural \xrightarrow{\mu_\phi} \bC_\bt(\cA_\Gamma)_{[\phi]}^\natural.
 \] 
There are analogous quasi-isomorphisms between the respective periodic cyclic complexes. This provides us with isomorphisms,
\begin{gather}
 \bHC_m(\cA_\Gamma)_{[\phi]} \simeq \left(\Omega^m(M^\phi)^{\Gamma_\phi}/d\Omega^{m+1}(M^\phi)^{\Gamma_\phi}\right) \oplus 
 H^{m-2}(M^\phi)^{\Gamma_\phi} \oplus \cdots, \qquad m\geq 0, 
 \label{eq:isom-HC-AG-manifold-finite1}\\
  \bHP_i (\cA_\Gamma)_{[\phi]}\simeq \bigoplus_{q\geq 0} H^{2q+i}(M^\phi)^{\Gamma_\phi}, \qquad i=0,1. 
   \label{eq:isom-HP-AG-manifold-finite2}
\end{gather}
\end{theorem}
  
When $\Gamma$ is finite, combining~(\ref{eq:isom-HC-AG-manifold-finite1})--(\ref{eq:isom-HP-AG-manifold-finite2}) with~(\ref{eq:LCA.splitting-HCHP}) 
enables us to express $\HC_\bt(\cA_\Gamma)$ and $\HP_\bt(\cA_\Gamma)$ in terms of the orbifold cohomology, 
\begin{equation*}
 H^\bt(M/\Gamma):= \bigoplus H^\bt(M^\phi)^{\Gamma_\phi}.
\end{equation*}
In particular, we recover the following result of Baum-Connes.

\begin{corollary}[{Baum-Connes~\cite{BC:CCDG}}]\label{cor:finite-Baum-Connes}
 Suppose that $\Gamma$ is finite. Then 
 \begin{equation*}
\bHP_\bt(\cA_\Gamma)\simeq H^{\ev/\odd}(M/\Gamma). 
\end{equation*}
\end{corollary}
   
%\subsection{$\bC(\cA_\Gamma)_{[\phi]}$ when $\phi$ has finite order}   

Suppose now that $\phi$ has finite order. For $\phi=1$ Connes~\cite{Co:Kyoto83,Co:NCG} constructed an explicit quasi-isomorphism from $C^{\ev/\odd}_\Gamma(M)$ to the periodic cyclic cochain complex of the homogeneous component $\bC(\cA_\Gamma)_{[1]}$. More generally, when $\phi$ has finite order, by combining Theorem~\ref{thm:LCA.bHCAGphi-finite-order} and Proposition~\ref{prop:twisted-CHKR} we obtain the following result.

\begin{theorem}\label{thm:finite-order} 
Let $\phi \in \Gamma$ have finite order. We have explicit quasi-isomorphisms, 
\begin{equation}
 \Tot_\bt\left( C(\Gamma_\phi,M^\phi)\right)^\natural  \xleftarrow[]{(\varepsilon \nu_\phi)\otimes \alpha^\phi}
 \Tot_\bt\left(\bC^\phi(\Gamma_\phi,\cA)\right)^\natural\xrightleftharpoons[\AW^\natural]{\shuffle^\natural} \Diag_\bt\left(\bC^\phi(\Gamma_\phi,\cA)\right)^\natural \xrightarrow{\mu_\phi} \bC_\bt(\cA_\Gamma)_{[\phi]}^\natural.
 \label{eq:manifolds.quasi-isom-finite-order}
\end{equation}
  There are similar quasi-isomorphisms between the respective periodic cyclic complexes. This gives rise to isomorphisms,
\begin{equation}
\bHC(\cA_\Gamma)_{[\phi]}\simeq H^{\Gamma_\phi}_\bt(M^\phi)^\natural \quad  \text{and} \quad \bHP(\cA_\Gamma)_{[\phi]}\simeq H^{\Gamma_\phi}_{\ev/\odd}(M^\phi)^\sharp.
\label{eq:HCHPAMGphi} 
\end{equation}
\end{theorem}

\begin{remark}
 Brylinski-Nistor~\cite{BN:KT94} (see also Crainic~\cite{Cr:KT99}) expressed  $\bHC_\bt(\cA_\Gamma)_{[\phi]}$ and $\bHP_\bt(\cA_\Gamma)_{[\phi]}$ in terms of the equivariant homology of $M^\phi$. We obtain explicit quasi-isomorphisms with the equivariant homology complex by combining the quasi-isomorphisms~(\ref{eq:manifolds.quasi-isom-finite-order}) with the Poincar\'e duality for the de Rham complex $\Omega(M^\phi)$. In particular, this enables us to recover the aforementioned results of~\cite{BN:KT94}.
\end{remark}

\begin{remark}
 As mentioned above Connes~\cite{Co:Kyoto83,Co:NCG} constructed an explicit cochain map and quasi-isomorphism from $C^{\ev/\odd}_{\Gamma}(M)$ to the periodic cyclic cochain complex of the homogeneous component $\bC(\cA_\Gamma)_{[1]}$.   We also obtain from Theorem~\ref{thm:finite-order} a chain map at the (non-periodic) level. More precisely, when $\phi=1$ the embedding $\mu_\phi$ is actually an isomorphism of cyclic modules. Therefore, we obtain a chain map and quasi-isomorphism, 
  \begin{equation*}
 \Tot_\bt\left( C(\Gamma,M)\right)^\natural  \xrightarrow{(\varepsilon \otimes \alpha^1)\circ \AW^\natural\circ  \mu^{-1}} \bC_\bt(\cA_\Gamma)_{[1]}^\natural. 
% \label{eq:manifolds.quasi-isom-finite-order}
\end{equation*}
where $\mu^{-1}:\bC_\bt(\cA_\Gamma)_{[1]}\rightarrow \Diag_\bt(\bC(\Gamma,\cA))$ is given by~(\ref{eq:splitting.inverse-mu}). 
\end{remark}

Let $\eta^\phi: H^{\Gamma_\phi}_{\ev/\odd}(M^\phi)^\sharp\rightarrow \bHP(\cA_\Gamma)_{[\phi]}$ be the isomorphism defined by the quasi-isomorphisms~(\ref{eq:manifolds.quasi-isom-finite-order}). Composing it with the cap product~(\ref{eq:cap-products-evenodd}) provides us with the following corollary.

\begin{corollary}\label{cor:finite-order.cap-HG-HP}
 Let $\phi \in \Gamma$ have finite order. Then we have a bilinear graded map, 
\begin{equation*}
\eta^\phi(- \frown-) : H^{\ev/\odd}_{\Gamma_\phi}(M^\phi) \times H_{\ev/\odd}(\Gamma_\phi,\C) \longrightarrow \bHP_\bt(\cA_\Gamma)_{[\phi]}. 
%\label{eq:cap-products-evenodd-eta}
\end{equation*}
In particular, equivariant characteristic classes naturally give rise to classes in $\bHP_\bt(\cA_\Gamma)$. 
\end{corollary}

 The definition of the isomorphism $\eta^\phi$  involves the bi-paracyclic versions of the shuffle and Alexander-Whitney maps. These maps are given by explicit formulas, but the formulas involve many terms. We shall now see that we actually get a very simple formula when we pair $\eta^\phi$ with cocycles arising from equivariant currents. 

Let $\Omega^\Gamma(M)=(\Omega^\Gamma_\bt(M), d)$ be the cochain complex of equivariant currents, where 
$\Omega^\Gamma_m(M)$, $m\geq 0$, consists of maps $C:\Gamma\rightarrow \Omega_m(M)$ that are $\Gamma$-equivariant in the sense that
$C(\psi_1^{-1} \psi_0 \psi_1)=(\psi_1)_*[C(\psi_0)]$ for all  $\psi_j \in \Gamma$. (Here $\Omega_m(M)$ is the space of $m$-dimensional currents.)  Any equivariant current $C\in \Omega^\Gamma_m(M)$ defines a cochain $\varphi_C\in C^m(\cA_\Gamma)$ by the formula, 
\begin{equation*}
\varphi_C(f^0u_{\psi_0}, \ldots, f^mu_{\psi_m})= \frac{1}{m!} \acou{C(\psi)}{f^0 d\hat{f}^1 \wedge \cdots \wedge d\hat{f}^m}, \qquad f^j \in \cA, \ \phi_j\in \Gamma,
\end{equation*}
where we have set $\psi=\psi_0 \cdots \psi_m$ and $\hat{f}^j= f^j\circ (\psi_{0} \cdots\psi_{j-1})^{-1}$. This provides us with a map of mixed complexes  from 
$(\Omega_\bt^\Gamma(M), d,0)$ to $(C^\bt(\cA_\Gamma), B, b)$. Therefore, we obtain cochain maps from between the associated cyclic (resp., periodic) cochain complexes. 
Note that the periodic complex of $\Omega^\Gamma(M)$ is just $(\Omega^\Gamma_{\ev/\odd}(M), d)$. The transverse fundamental class of Connes~\cite{Co:Kyoto83} and the CM cocycle of an equivariant Dirac spectral triple~\cite{PW:JNCG16} are instances of cocycles arising from equivariant currents. 

In what follows, given any equivariant chain $\omega=(\omega_{p,q})$, $\omega_{p,q}\in C_{p,q}(\Gamma,M)$, we shall denote by $\omega_0$ its component in $C_{0,\bt}(\Gamma, M)=\C\Gamma \otimes_{\Gamma} \Omega^\bt(M)\simeq \Omega^\bt(M)$.

\begin{proposition}\label{prop:manifolds.pairing-equivarian-currents}
Let $\phi \in \Gamma$ have finite order. Then, for any closed equivariant current $C\in \Omega_{\ev/\odd}^\Gamma(M)$ and any equivariant cycle $\omega \in C_{\ev/\odd}(\Gamma_\phi, M^\phi)$, we have 
\begin{equation*}
\acou{\varphi_C}{\eta^\phi(\omega)}= \acou{C(\phi)}{\tilde{\omega}_0}, 
\end{equation*}
where $\tilde{\omega}_0\in \Omega^\ev(M)$ is such that $\tilde{\omega}_{0|M^\phi}=\omega_0$. 
\end{proposition}

For instance, let $E$ be a $\Gamma_\phi$-equivariant vector bundle over a submanifold component $M^\phi_a$. Given any connection $\nabla^E$ on $E$, the equivariant Chern character of $E$ is represented by an equivariant cocycle $\Ch_{\Gamma_\phi}(\nabla^E)\in C^\ev_{\Gamma_\phi}(M^\phi_a)$ defined as follows. It has components $\Ch^p_{\Gamma_\phi}(\nabla^E) \in C^p(\Gamma_\phi, \Omega^{\bt}(M^\phi_a))$ given by 
\begin{gather*}
\Ch^0_{\Gamma_\phi}(\nabla^E)(\psi_0)=\Ch\left((\psi_0)_*\nabla^E\right), \\ \Ch^p_{\Gamma_\phi}(\nabla^E)(\psi_0,\ldots, \psi_p)=
(-1)^p \CS \left((\psi_0)_*\nabla^E, \ldots, (\psi_p)_*\nabla^E\right), \quad \psi_j\in \Gamma_\phi. 
\end{gather*}
 Here  $\Ch$ is the usual Chern character and $\CS((\psi_0)_*\nabla^E, \ldots, (\psi_p)_*\nabla^E)$ is the Chern-Simons form of the connections $(\psi_0)_*\nabla^E, \ldots, (\psi_p)_*\nabla^E$ as defined in~\cite{Ge:AIM94}. Note that $\CS((\psi_0)_*\nabla^E, \ldots, (\psi_p)_*\nabla^E)$ is an element of $C^p(\Gamma_\phi, \Omega^\ev(M^\phi_a))$ 
 (resp., $C^p(\Gamma_\phi, \Omega^\ev(M^\phi_a))$) when $p$ is even (resp., odd). 
 
 Combining Proposition~\ref{prop:manifolds.pairing-equivarian-currents} and Corollary~\ref{cor:finite-order.cap-HG-HP} shows that, for any closed equivariant current $C$ in 
 $\Omega_{\ev/\odd}^\Gamma(M)$ such that $\op{supp}C(\phi)\subset M^\phi$ and for any cycle $\xi\in C_p(\Gamma_\phi, \C)$, we have
 \begin{equation*}
\acou{\varphi_C}{\eta^\phi (\Ch_\Gamma(\nabla^E) \frown \xi)} = \acou{C(\phi)}{(\Ch_\Gamma(\nabla^E) \frown \xi)_0}. 
\end{equation*}
Here $(\Ch_\Gamma(\nabla^E) \frown \xi)_0$ is simple to calculate. If we set $\xi= \sum_\ell \lambda_\ell (\psi_0^\ell,  \ldots, \psi_p^\ell)\otimes_{\Gamma_\phi} 1$, 
then we have 
\begin{equation*}
\left(\Ch_\Gamma(\nabla^E) \frown \xi \right)_0= \sum_\ell (-1)^p  \CS \left((\psi_0^\ell)_*\nabla^E, \ldots, (\psi_p^\ell)_*\nabla^E\right). 
\end{equation*}
In particular, for the generator $1: =1 \otimes_{\Gamma_\phi} 1$ of $C_0(\Gamma_\phi,\C)$, we get
\begin{equation*}
\acou{\varphi_C}{\eta^\phi (\Ch_\Gamma(\nabla^E) \frown 1)} = \acou{C(\phi)}{\Ch(\nabla^E)}. 
\end{equation*}

Finally, suppose that $\phi$ has infinite order and acts cleanly on $M$. Set $\OG_\phi=\Gamma_\phi / \brak\phi$.  As $\Omega(M^\phi)$ is a $\phi$-invariant mixed complex, we may form the triangular $S$-module $C^\sigma(\OG_\phi,\Omega(M^\phi))$  and its total $S$-module $\Tot(C^\sigma(\OG_\phi, \Omega(M^\phi)))$ as in Section~\ref{sec:infinite-order}. We observe that the space of $m$-chains and the differential of $\Tot(C^\sigma(\OG_\phi, \Omega(M^\phi)))$ are given by 
\begin{equation*}
\Tot_m(C^\sigma(\OG_\phi, \Omega(M^\phi))) =\bigoplus_{p+q=m} C_p(\OG_\phi, \Omega^q(M^\phi)), \qquad d^\dagger = \partial + (-1)^p d(u_\phi \frown -). 
\end{equation*}
Here $u_\phi\in C^2(\OG_\phi, \C)$ is any cocycle representing the Euler class $e_\phi \in H^2(\OG_\phi,\C)$ of the extension of $\OG_\phi$ by $\Gamma_\phi$. 
By combining Theorem~\ref{thm:LCA.infinite-good-action} and Proposition~\ref{prop:twisted-CHKR} we then obtain the following result.

\begin{theorem}\label{thm:manifolds.infinite-order-clean}
 Let $\phi \in \Gamma$ have infinite order and act cleanly on $M$.  We have the following quasi-isomorphisms of chain complexes,
\begin{equation}
 \Tot_\bt (C^\sigma(\overline{\Gamma}_\phi ,\Omega(M^\phi)))  \xleftarrow{\theta(1\otimes \alpha^\phi)}
 \Tot_\bt\left(\bC^\phi(\Gamma_\phi,\cA)\right)^\natural\xrightleftharpoons[\AW^\natural]{\shuffle^\natural} \Diag_\bt\left(\bC^\phi(\Gamma_\phi,\cA)\right)^\natural \xrightarrow{\mu_\phi} \bC_\bt(\cA_\Gamma)_{[\phi]}^\natural. 
 \label{eq:Manifolds.quasi-isom-infinite-clean}
\end{equation}
This gives rise to the isomorphism,  
 \begin{equation}
\bHC(\cA_\Gamma)_{[\phi]}\simeq H_\bt \left(  \Tot_\bt (C^\sigma(\overline{\Gamma}_\phi ,\Omega(M^\phi)))\right). 
\end{equation}
Under this isomorphism, the periodicity operator of $\HC_\bt(\cA)_{[\phi]}$ is given by the cap product, 
\begin{equation*}
e_\phi \frown - : H_\bt\left(\Tot(C^\sigma(\overline{\Gamma}_\phi ,\Omega(M^\phi)))\right) \longrightarrow H_{\bt-2}(\Tot(C^\sigma(\overline{\Gamma}_\phi ,\sC))). 
\end{equation*}
\end{theorem}

In this context, Corollary~\ref{cor:LCA.infinite-spectral-sequence-good} gives the following statement.

\begin{corollary}\label{cor:manifolds.infinite-order-clean-spectral-sequence}
  Let $\phi \in \Gamma$ have infinite order and act cleanly on $M$. Then the quasi-isomorphisms~(\ref{eq:Manifolds.quasi-isom-infinite-clean}) and the filtration by columns of $\Tot_\bt(C^\sigma(\overline{\Gamma}_\phi ,\Omega(M^\phi)))$ give rise to a spectral sequence, 
  \begin{equation}
E^2_{p,q}=H_p(\OG_\phi, \Omega^q(M^\phi)) \Longrightarrow \bHC_{p+q}(\cA_\Gamma).
\label{eq:manifolds.infinite-order-clean-spectral-sequence}
\end{equation}
Moreover, the $E_2$-differential is $(-1)^pd(u_\phi\frown -): H_p(\OG_\phi, \Omega^q(M^\phi))\rightarrow H_{p-2}(\OG_\phi, \Omega^{q+1}(M^\phi))$. 
\end{corollary}

\begin{remark}\label{rmk:mfld-infinite.Crainic}
Crainic obtained a spectral sequence like~(\ref{eq:manifolds.infinite-order-clean-spectral-sequence}). He then inferred from this that  $\bHC_m(\cA)_{[\phi]} \simeq \oplus_{p+q=m} H_p(\OG_\phi, \Omega^q(M^\phi))$ (see~\cite[Corollary~4.15]{Cr:KT99}). What we really have is the isomorphism $\bHC_\bt(\cA)_{[\phi]}\simeq H_\bt\left( \Tot(C^\sigma(\overline{\Gamma}_\phi ,\Omega(M^\phi)))\right)$ given by Theorem~\ref{thm:manifolds.infinite-order-clean}.
\end{remark}

 \section{Group Actions on Smooth Varieties}\label{sec:varieties} 
In this section, we explain how to obtain analogues of the results of the previous section to group actions on smooth varieties. Throughout this section we assume that $\Gamma$ is a group that acts by biregular isomorphisms on a smooth variety $X$. This action then gives an action on the algebra $\cA=\cO(X)$ of regular functions on $X$. In what follows, given any $\phi \in \Gamma$, we shall denote by $X^\phi$ the fixed-point set of the action of $\phi$ on $X$. 

\begin{definition}[see~\cite{BDN:AIM17}] 
  Let $\phi\in\Gamma$. We shall say that the action of $\phi$ on $X$ is \emph{clean} when the following conditions are satisfied: 
\begin{enumerate}
\item $X^\phi$ is a smooth variety. 

\item $T_{x_0}X^\phi= \ker (\phi'(x_0)-1)$ and $T_{x_0}X^\phi \cap \ran (\phi'(x_0)-1)=\{0\}$ for all $x_0\in X^\phi$.
\end{enumerate}
\end{definition}

\begin{remark}
 When $\phi$ has finite order its action on $X$ is always clean. 
\end{remark}

Let $\phi \in \Gamma$ act cleanly. As $X^\phi$ is a smooth variety we may form the de Rham mixed complex $(\Omega^\bt(X^\phi), 0, d)$ of algebraic forms on $X^\phi$. The twisted Hochschild-Kostant-Rosenberg map $\alpha^\phi: C^\phi(\cA)\rightarrow \Omega^\bt(X^\phi)$ is then given by 
\begin{equation}
\alpha^\phi(a^0\otimes \cdots \otimes a^m)=\frac{1}{m!} a^0 da^1 \cdots da^m, \qquad a^j\in \cA.
\label{eq:HKR-map-variety}
\end{equation}
This is a map of parachain complexes. Furthermore, we have the following result.

\begin{proposition}[{Brodzki-Dave-Nistor~\cite{BDN:AIM17}}] \label{prop:twistedHKR}
 Let $\phi \in \Gamma$ act cleanly on $X$. Then the parachain complex map~(\ref{eq:HKR-map-variety}) is a quasi-isomorphism. 
\end{proposition}

\begin{remark}
 When $\phi=1$ the result is due to Loday-Quillen~\cite{LQ:CMH84} (see also~\cite{HKR:TAMS62}). 
\end{remark}
 
Combining Proposition~\ref{prop:twistedHKR} with the results of the Sections~\ref{sec:finite}--\ref{sec:infinite-order} enables us to determine the cyclic homology of the crossed-product $\cA_\Gamma=\cA\rtimes \Gamma$ in the same way as with group actions on manifolds in the previous section. 

First, when $\Gamma_\phi$ is finite we have the following result.

\begin{theorem}\label{thm:varieties.finite}
 Let $\phi\in \Gamma_\phi$ have a finite centralizer $\Gamma_\phi$. We have explicit quasi-isomorphisms, 
\[
  \Omega(X^\phi)^{\Gamma_\phi,\natural}_\bt  \xleftarrow{\pi_0\otimes \alpha^\phi}
 \Tot_\bt\left(C^\phi(\Gamma_\phi,\cA)\right)^\natural\xrightleftharpoons[\AW^\natural]{\shuffle^\natural} \Diag_\bt\left(C^\phi(\Gamma_\phi,\cA)\right)^\natural \xrightarrow{\mu_\phi} C_\bt(\cA_\Gamma)_{[\phi]}^\natural.
 \] 
There are similar quasi-isomorphisms at the level of the periodic cyclic complexes. This provides us with isomorphisms,
\begin{gather}
 \HC_m(\cA_\Gamma) \simeq \left(\Omega^m(X^\phi)^{\Gamma_\phi}/d\Omega^{m+1}(X^\phi)^{\Gamma_\phi}\right) \oplus 
 H^{m-2}(X^\phi)^{\Gamma_\phi} \oplus \cdots, \qquad m\geq 0, 
 \label{eq:isom-HC-AG-variety-finite}\\
  \HP_i (\cA_\Gamma)\simeq \bigoplus_{q\geq 0} H^{2q+i}(X^\phi)^{\Gamma_\phi}, \qquad i=0,1. 
   \label{eq:isom-HP-AG-variety-finite}
\end{gather}
\end{theorem}
 
\begin{remark}
 Brodzki-Dave-Nistor~\cite[Theorem~2.19]{BDN:AIM17} obtained the isomorphisms~(\ref{eq:isom-HC-AG-variety-finite})--(\ref{eq:isom-HP-AG-variety-finite}),   
but they did not exhibit explicit quasi-isomorphisms. They also extended the isomorphism~(\ref{eq:isom-HP-AG-variety-finite}) to the case where $X$ is non-smooth.
\end{remark}

\begin{remark}
 When $\Gamma$ is finite, combining~(\ref{eq:isom-HC-AG-variety-finite})--(\ref{eq:isom-HP-AG-variety-finite}) with~(\ref{eq:splitting-HCHP}) 
enables us to express $\HC_\bt(\cA_\Gamma)$ and $\HP_\bt(\cA_\Gamma)$ in terms of the (algebraic) orbifold cohomology $H^\bt(X/\Gamma):= \oplus H^\bt(X^\phi)^{\Gamma_\phi}$ (\emph{cf.}~\cite{BDN:AIM17}). 
This is the algebraic version of Corollary~\ref{cor:finite-Baum-Connes}. 
\end{remark}

More generally, as with group actions on manifolds, when $\phi$ has finite order we can relate the cyclic homology and periodic cyclic homology of $C(\cA_\Gamma)_{[\phi]}$ to mixed equivariant homology and equivariant homology for smooth varieties as follows. 

We define the \emph{equivariant cohomology} $H^\bt_{\Gamma_\phi}(X^\phi)$ as the cohomology of the total complex of the bicomplex 
$C_{\Gamma_\phi}(X^\phi)=(C_{\Gamma_\phi}^{\bt,\bt}(X^\phi), \partial, d)$, where $C^{p,q}_{\Gamma_\phi}(X^\phi)$ consists of all $\Gamma_\phi$-equivariant maps $\omega:\Gamma_\phi^{p+1}\rightarrow \Omega^q(X^\phi)$. We define the \emph{even/odd equivariant cohomology} $H^{\ev/\odd}_{\Gamma_\phi}(X^\phi)$ as the cohomology of the complex $(C^{\ev/\odd}_{\Gamma_\phi}(X^\phi), \partial +(-1)^pd)$, where ${ C^{\ev/\odd}_{\Gamma_\phi}(X^\phi)=\prod_{\text{$p\!+\!q$\! even/odd}} C^{p,q}_{\Gamma_\phi}(X^\phi)}$. 

The \emph{mixed equivariant bicomplex} $C(\Gamma_\phi, X^\phi)$ is the mixed bicomplex $(C_{\bt,\bt}(\Gamma_\phi, X^\phi),\partial, 0,0,d)$, where 
$C_{p,q}(\Gamma_\phi, X^\phi)=C_p(\Gamma_\phi)\otimes_{\Gamma_\phi} \Omega^q(X^\phi)$. The \emph{mixed equivariant homology} $H^{\Gamma_\phi}_\bt(X^\phi)^\natural$ is the cyclic homology of the total mixed complex $\Tot(C(\Gamma_\phi, X^\phi))=( \Tot_\bt(C(\Gamma_\phi, X^\phi)),\partial, (-1)^pd)$. The \emph{even/odd mixed equivariant homology} $H^{\Gamma_\phi}_{\ev/\odd}(X^\phi)^\sharp$ is the periodic cyclic homology of $\Tot(C(\Gamma_\phi, X^\phi))$. In the same way as in Section~\ref{sec:equivariant}, we have a natural (graded) cap product, 
 \begin{equation}
  \frown :  H^{\ev/\odd}_{\Gamma_\phi}(X^\phi) \times H^{\ev/\odd}(\Gamma_\phi, \C) \longrightarrow H^{\Gamma_\phi}_{\ev/\odd}(X^\phi)^\sharp.
  \label{eq:cap-products-evenodd-varieties}
\end{equation}

Bearing all this in mind, in the same way as with group actions on manifolds in Section~\ref{sec:manifolds}, the mixed equivariant bicomplex $C(\Gamma_\phi, X^\phi)$ is the 
mixed bicomplex $C^\flat(\Gamma_\phi, \sC^\phi)$ for $\sC=\Omega(X^\phi)$. Therefore, by combining Theorem~\ref{thm:finite-order.HCAGphi} with Proposition~\ref{prop:twistedHKR} we obtain the following result. 

\begin{theorem}\label{thm:varieties.finite-order}
 Let $\phi \in \Gamma$ have finite order. Then we have explicit quasi-isomorphisms, 
 \begin{equation}
 \Tot_\bt\left(C(\Gamma_\phi, X^\phi)\right)^\natural  \xleftarrow{(\varepsilon \nu_\phi)\otimes \alpha^\phi}
 \Tot_\bt\left(C^\phi(\Gamma_\phi,\cA)\right)^\natural\xrightleftharpoons[\AW^\natural]{\shuffle^\natural} \Diag_\bt\left(C^\phi(\Gamma_\phi,\cA)\right)^\natural \xrightarrow{\mu_\phi} C_\bt(\cA_\Gamma)_{[\phi]}^\natural.
\label{eq:varieties.quasi-isom-finite-order}
\end{equation}
There are analogous quasi-isomorphisms between the corresponding periodic cyclic complexes. This provides us with isomorphisms,
\[
\HC_\bt(\cA_\Gamma)_{[\phi]}\simeq H^{\Gamma_\phi}_\bt(X^\phi)^\natural \qquad \text{and} \qquad 
\HP_\bt(\cA_\Gamma)_{[\phi]}\simeq H^{\Gamma_\phi}_\bt(X^\phi)^\sharp.
\]
\end{theorem}

Let $\eta^\phi: H^{\Gamma_\phi}_{\ev/\odd}(X^\phi)^\sharp\rightarrow \HP_\bt(\cA_\Gamma)_{[\phi]}$ be the isomorphism defined by the quasi-isomorphisms~(\ref{eq:varieties.quasi-isom-finite-order}). In the same way as with group actions on manifolds, composing it with the cap product~(\ref{eq:cap-products-evenodd-varieties}) we arrive at the following statement.

\begin{corollary}%\label{cor:finite-order.cap-HG-HP-varieties}
 Let $\phi \in \Gamma$ have finite order. Then we have a bilinear graded map, 
\begin{equation*}
\eta^\phi(- \frown-) : H^{\ev/\odd}_{\Gamma_\phi}(X^\phi) \times H_{\ev/\odd}(\Gamma_\phi,\C) \longrightarrow \HP_\bt(\cA_\Gamma)_{[\phi]}. 
\end{equation*}
In particular, equivariant characteristic classes in $H^{\ev/\odd}_{\Gamma_\phi}(X^\phi)$ give rise to classes in $\bHP_\bt(\cA_\Gamma)_{[\phi]}$. 
\end{corollary}

Finally, suppose that $\phi$ has infinite order and acts cleanly on $X$. Set $\OG_\phi=\Gamma_\phi / \brak\phi$.  As in Section~\ref{sec:manifolds}, $\Omega(X^\phi)$ is a $\phi$-invariant mixed complex, and so we may form the triangular $S$-module $C^\sigma(\OG_\phi,\Omega(X^\phi))$ and its total $S$-module $\Tot(C^\sigma(\OG_\phi, \Omega(X^\phi)))$. The space of $m$-chains of $\Tot(C^\sigma(\OG_\phi, \Omega(X^\phi)))$ and its differential are then given by
\begin{equation*}
\Tot_m(C^\sigma(\OG_\phi, \Omega(X^\phi))) =\bigoplus_{p+q=m} C_p(\OG_\phi, \Omega^q(X^\phi)), \qquad d^\dagger = \partial + (-1)^p d(u_\phi \frown -). 
\end{equation*}
Here $u_\phi\in C^2(\OG_\phi, \C)$ is any cocycle representing the Euler class $e_\phi \in H^2(\OG_\phi,\C)$ of the extension of $\OG_\phi$ by $\Gamma_\phi$. 
By combining Theorem~\ref{thm:infinite.HCAGphi} and Proposition~\ref{prop:twistedHKR} we then obtain the following result.

\begin{theorem}\label{thm:varieties.infinite-order}
 Let $\phi \in \Gamma$ have infinite order and act cleanly on $X$.  We have the following quasi-isomorphisms of chain complexes, 
  \begin{equation}
 \Tot_\bt (C^\sigma(\overline{\Gamma}_\phi ,\Omega(X^\phi)))  \xleftarrow{\theta(1\otimes \alpha)}
 \Tot_\bt\left(C^\phi(\Gamma_\phi,\cA)\right)^\natural\xrightleftharpoons[\AW^\natural]{\shuffle^\natural} \Diag_\bt\left(C^\phi(\Gamma_\phi,\cA)\right)^\natural \xrightarrow{\mu_\phi} C_\bt(\cA_\Gamma)_{[\phi]}^\natural.
 \label{eq:varieties.infinite-quasi-isom}
\end{equation} 
This provides us with an isomorphism,
\begin{equation*}
\HC_\bt(\cA)_{[\phi]}\simeq H_\bt\left( \Tot(C^\sigma(\overline{\Gamma}_\phi ,\Omega(X^\phi)))\right).
%\label{eq:infinite.isom-HCAG-HTotCsGsC-variety} 
\end{equation*}
 Under this isomorphism the periodicity operator of $\HC_\bt(\cA)_{[\phi]}$ is given by the cap product, 
\begin{equation*}
e_\phi \frown - : H_\bt\left(\Tot(C^\sigma(\overline{\Gamma}_\phi ,\Omega(X^\phi)))\right) \longrightarrow H_{\bt-2}\left(\Tot(C^\sigma(\overline{\Gamma}_\phi ,\Omega(X^\phi)))\right),
\end{equation*}
\end{theorem}

In this context, Corollary~\ref{cor:infinite.FT-spectral-sequence-good} gives the following statement. 
 
 \begin{corollary}
Let $\phi \in \Gamma$ have infinite order and act cleanly on $X$. Then the quasi-isomorphisms~(\ref{eq:varieties.infinite-quasi-isom}) and the filtration by columns of $\Tot_\bt(C(\OG_\phi)\otimes_{\OG_\phi} \sC)$ give rise to a spectral sequence,
\begin{equation*}
E^2_{p,q}= H_p(\OG_\phi, \Omega^q(X^\phi)) \Longrightarrow \HC_{p+q}(\cA_\Gamma)_{[\phi]}. 
\end{equation*}
 Moreover, the $E^2$-differential  is $(-1)^pd(u_\phi \frown -): H_p(\OG_\phi, \Omega^q(X^\phi)) \rightarrow H_{p-2}(\OG_\phi, \Omega^{q+1}(X^\phi))$. 
\end{corollary}

\begin{remark}
 By using the results of Section~\ref{sec:LCA} the results of this sections can be extended to the crossed-products associated with the action on the $I$-adic completions of $\cA=\cO(X)$ considered in~\cite{BDN:AIM17}. The results are expressed in terms of the $I$-adic completions of the de Rham complex $\Omega(X^\phi)$ as defined in~\cite{BDN:AIM17}. 
\end{remark}

\end{document}